\documentclass[11pt,a4paper]{amsart}
\usepackage{latexsym}
\usepackage[colorlinks]{hyperref}
\usepackage{xcolor}
\usepackage{mathscinet}
\usepackage{amsthm}
\usepackage{amssymb}
\usepackage{amsfonts}
\usepackage{amsmath}
\title[Existence and Exponential mixing]
{Existence and Exponential mixing of infinite white $\alpha$-stable
Systems with unbounded interactions}
\author[L. Xu]{Lihu Xu}
\address{PO Box 513, EURANDOM, 5600 MB  Eindhoven. The Netherlands}
\email{xu@eurandom.tue.nl}
\author[B.Zegarlinski]{Bogus{\l}aw Zegarli\'nski}
 \address{CNRS, Toulouse (on leave of absence from
Mathematics Department, Imperial College London, SW7 2AZ, United
Kingdom)}
\email{zegarlinski@univ-toulouse.fr}
\date{}
 \newtheorem{thm}{Theorem}[section]
 
 \newtheorem{lem}[thm]{Lemma}
 \newtheorem{prop}[thm]{Proposition}
 \newtheorem{assumption}[thm]{Assumption}
 \theoremstyle{definition}
 \newtheorem{defn}[thm]{Definition}
 \theoremstyle{remark}
 \newtheorem{rem}[thm]{Remark}
 \numberwithin{equation}{section}
 \newcommand{\R}{\mathbb{R}}

 \newcommand{\p}{\partial}
 \newcommand{\D}{\mathcal{D}}
 \newcommand{\e}{\varepsilon}
 \newcommand{\Z}{\mathbb{Z}}
 \newcommand{\E}{\mathbb{E}}
 \newcommand{\N}{\mathbb{N}}
 \newcommand{\mcl}{\mathcal}

 \newcommand{\est}{\mcl E(s,t)}
 \newcommand{\eot}{\mcl E(0,t)}
 \newcommand{\estn}{\mcl E^{(n)}(s,t)}

\begin{document}
\maketitle
%
%\pagebreak
%\vspace{5cm}

\begin{abstract} \label{abstract}
We study an infinite white $\alpha$-stable systems with unbounded interactions, proving the existence
by Galerkin approximation and exponential mixing property by an $\alpha$-stable version of gradient bounds.
\ \\

\noindent {\it Key words and phrases:} Ergodicity,
White symmetric $\alpha$-stable processes, Lie bracket, Finite speed of propagation of information, Gradient bounds.

\noindent {\it 2000 Mathematics Subject Classification.} 37L55, 60H10, 60H15.
\end{abstract}
\section{Introduction}
The SPDEs driven by L$\acute{e}$vy noises were intensively
studied in the past several decades (\cite{PeZa07}, \cite{AMR09},\cite{PeZa06}, \cite{PrZa09-2}, \cite{BaCh06},
\cite{AWZ98}, \cite{Ok08}, \cite{MaRo09}, $\cdots$). The noises can be Wiener(\cite{DPZ92},\cite{DPZ96}) Poisson (\cite{AWZ98}), $\alpha$-stable types (\cite{PrZa09},\cite{XZ09}) and so on.
To our knowledge, many of these results in these articles are in the frame of Hilbert space,
and thus one usually needs to
assume that the L$\acute{e}$vy noises are square integrable. This assumption rules out
a family of important L$\acute{e}$vy noises -- $\alpha$-stable noises. On the other hand, the ergodicity of SPDEs has also been intensively studied recently (\cite{DPZ96},\cite{HM06}, \cite{RX09},  \cite{XZ09}, \cite{FuXi09}), most of
these known results are about the SPDEs driven by Wiener type noises. There exist few results on the ergodicity of the SPDEs driven by the jump noises (\cite{XZ09}, \cite{PeZa07}). \\

In this paper, we shall study an interacting spin system driven by \emph{white} symmetric
$\alpha$-stable noises ($1<\alpha \leq 2$). More precisely, our system is described by the following
infinite dimensional SDEs: for each $i \in \mathbb{Z}^d$,
\begin{equation} \label{e:IntSys}
\begin{cases}
dX_i(t)=[J_i(X_i(t))+I_i(X(t))]dt+dZ_i(t) \\
X_i(0)=x_i
\end{cases}
\end{equation}
where $X_i, x_i \in \R$, $\{Z_i; i \in \Z^d\}$ are a sequence of i.i.d.
symmetric $\alpha$-stable
processes with $1<\alpha \leq 2$, and the assumptions for the $I$ and $J$ are specified in Assumption
\ref{a:IJ}.
Equation \eqref{e:IntSys} can be considered as a SPDEs in some Banach space, we shall
study the existence of the dynamics, Markov property and the exponential mixing property. When $Z(t)$ is Wiener noise, the equation \eqref{e:IntSys} has been intensively studied in modeling quantum spin systems in the 90s of last century (see e.g. \cite{AKYR95}, \cite{AKYT94}, \cite{DPZ96}, $\cdots$). \emph{Besides} this, we have the other two motivations to study \eqref{e:IntSys} as follows. \\

The first motivation is to extend the known existence and ergodic results about the interacting system in Chapter 17 of \cite{PeZa07}. In that book, some interacting systems similar to \eqref{e:IntSys} were studied under the framework of SPDEs (\cite{DPZ92}, \cite{DPZ96}). In order to prove the existence and ergodicity, one needs to assume that the noises are square integrable and that the interactions are linear and finite range. Comparing with the systems in \cite{PeZa07}, the white $\alpha$-stable noises in \eqref{e:IntSys} are \emph{not} square integrable, the interactions $I_i$ are not linear but \emph{Lipschitz} and have \emph{infinite range}. Moreover, we shall not work on Hilbert space but on some considerably large subspace $\mathbb B$ of $\R^{\Z^d}$, which seems more \emph{natural} (see Remark \ref{r:RemB}). The advantage of using this subspace is that we can split it into compact balls (under product topology) and control
some important quantities in these balls (see Proposition \ref{p:AppExiDif} for instance). Besides the techniques in
SPDEs, we shall also use those in interacting particle systems such as finite speed of propagation of information property.  \\

The second motivation is from the work by
 \cite{Ze96} on interacting unbounded spin systems driven by Wiener noise. The system studied there is also similar to \eqref{e:IntSys}, but has two essential differences. \cite{Ze96} studied a \emph{gradient} system perturbed by \emph{Wiener} noises, it is not hard to show the stochastic systems is reversible and admits a unique invariant measure $\mu$. Under the framework of $L^2(\mu)$,
the generator of the system is self-adjoint and thus we can construct dynamics by the spectral
decomposition technique. However, the deterministic part in \eqref{e:IntSys} is \emph{not} necessarily a gradient type and the noises are more general. This means that our system is possibly not reversible, so we have to construct the dynamics by some other method. More precisely, we shall prove the existence of the dynamics by studying some Galerkin approximation, and passing to its limit by the finite speed of propagation
and some uniform bounds of the approximate dynamics. On the other hand,
\cite{Ze96} proved the following pointwise ergodicity $|P_tf(x)-\mu(f)| \leq C(f,x) e^{-m t}$, where
$P_t$ is the semigroup generated by a reversible generator. The main tool for proving this ergodicity is by a logarithmic Sobolev inequality (LSI). Unfortunately, the LSI is not available in our setting, however, we can use the spirit of \emph{Bakry-Emery criterion} in LSI to obtain a \emph{gradient} bounds, from which we show the same ergodicity result as in \cite{Ze96}. We remark that although such strategy could be in principle applied to models considered in \cite{Ze96}, unlike the method based on LSI (where only asymptotic mixing is relevant), in the present level of technology it can only cover the weak interaction regime far from the `critical point'.
\ \\

Let us give two concrete examples for our system \eqref{e:IntSys}.
The first one is by setting $I_i(x)=\sum_{j \in \Z^d} a_{ij} x_j$
and $J_i(x_i)=-(1+\e)x_i-c x^{2n+1}_i$ with any $\e>0$, $c \geq
0$ and $n \in \N$ for all $i \in \Z^d$, where $(a_{ij})$ is a transition
probability of random walk on $\Z^d$. If we take $c=0$ and
$Z_i(t)=B_i(t)$ in \eqref{e:IntSys} with $(B_i(t))_{i \in \Z^d}$
i.i.d. standard Brownian motions, then this example is similar to
the neutral stepping stone model (see \cite{Da09}, or see a more
simple introduction in \cite{Xu09}) and the interacting diffusions
(\cite{GrHo07}, \cite{HuWa07}) in stochastic population dynamics. We
should point out that there are some essential differences between
these models and this example, but it is interesting to try our
method to prove the results in \cite{HuWa07}. The second example,
which has been introduced in \cite{OXZ08} in discrete dynamics, is
by setting $I_i(x)=\log\{\sum_{j \in \Z^d} a_{ji}e^{x_j}\}$ and
$J_i(x_i)=-(1+\e) x_i-c x_i^{2n+1}$, where $a_{ij}$, $\e$
and $c$ are the same as in the first example.   \\

The organization of the paper is as follows. Section 2 introduces some
notations and assumptions which will be used throughout the paper, and gives two key estimates.
In third and fourth sections, we shall prove the main theorems -- Theorem \ref{t:ConDyn} and Theorem \ref{t:Erg} respectively.
\\

{\bf Acknowledgements:} The first author would like to
thank the hospitality of Mathematics department of Universit$\acute{e}$
Paul Sabatier of Toulouse, part of his work was done during visiting Toulouse.
\section{Notations, assumptions, main results and two key estimates} \label{Introduction}
\subsection{Notations, assumptions and main results}
 We shall first introduce the definition of symmetric $\alpha$-stable processes
 ($0<\alpha \leq 2$), and then give more detailed description for the system \eqref{e:IntSys}. \\

Let $Z(t)$ be one dimensional $\alpha$-stable process ($0<\alpha \leq 2$), as $0<\alpha<2$, it has infinitesimal
generator $\partial^{\alpha}_x$ (\cite{ARW00}) defined by
\begin{equation} \label{e:fraclap}
\partial^{\alpha}_x f(x)=\frac{1}{C_{\alpha}} \int_{\mathbb{R} \setminus \{0\}} \frac{f(y+x)-f(x)}{|y|^{\alpha+1}}dy
\end{equation}
with $C_{\alpha}=
-\int_{\mathbb{R} \setminus \{0\}} (cosy-1)\frac{dy}{|y|^{1+\alpha}}.$  As $\alpha=2$, its generator is $\frac 12 \Delta$. One can also define $Z(t)$ by Poisson point processes or by Fourier transform (\cite{Be98}). The $\alpha$-stable property means
\begin{equation} \label{e:AlpStaPro}
Z(t) \stackrel{d}{=} t^{1/\alpha} Z(1).
\end{equation}
Note that we have use the symmetric property of $\p^\alpha_x$ in the easy identity $[\p^\alpha_x, \p_x]=0$ where $[\cdot, \cdot]$ is the Lie bracket. The \emph{white} symmetric $\alpha$-stable processes are defined by
$$\{Z_i(t)\}_{i \in \Z^d}$$
where $\{Z_i(t)\}_{i \in \Z^d}$ are a sequence of i.i.d. symmetric $\alpha$-stable process defined as the above.
\ \\

\indent We shall study the system \eqref{e:IntSys} on $\mathbb{B} \subset \mathbb{R}^{\mathbb{Z}^d}$ defined by $$\mathbb{B}=\bigcup \limits_{R>0, \rho>0}B_{R,\rho}$$
where for any $R,\rho>0$
 $$\ B_{R,\rho}=\{x=(x_i)_{i \in \Z^d}; |x_i| \leq R (|i|+1)^\rho\} \ \ {\rm with} \ \ |i|=\sum \limits_{k=1}^d |i_k|.$$

\begin{rem} \label{r:RemB}
The above $\mathbb B$ is a considerably large subspace of $\R^{\Z^d}$. Define the subspace $l_{-\rho}:= \{x \in \R^{\Z^d}; \  \sum_{k \in \Z^d} |k|^{-\rho} |x_k|<\infty\},$
it is easy to see that $l_{-\rho} \subset \mathbb B$ for all $\rho>0$. Moreover, one can also check that the distributions of the white $\alpha$-stable processes $(Z_i(t))_{i \in \Z^d}$ at any fixed time $t$ are supported on $\mathbb B$. From the form of the equation \eqref{e:IntSys}, one can expect that the distributions of the system at any fixed time $t$ is similar to those of white $\alpha$-stable processes but with some (complicated) shifts. Hence, it is \emph{natural} to study \eqref{e:IntSys} on $\mathbb B$.
\end{rem}
\begin{assumption} [Assumptions for $I$ and $J$] \label{a:IJ}
The $I$ and $J$ in \eqref{e:IntSys} satisfies the following conditions:
\begin{enumerate}
\item For all $i \in \Z^d$, $I_i: \mathbb B \longrightarrow \R$ is a continuous function
under the product
topology on $\mathbb B$ such that
\begin{equation*} \label{e:AssI}
|I_i(x)-I_i(y)| \leq \sum_{j \in \Z^d} a_{ji} |x_j-y_j|
\end{equation*}
where $a_{ij} \geq 0$ satisfies the conditions: $\exists$ some constants $K,K^{'},\gamma>0$ such that as $|i-j| \geq K^{'}$
\begin{equation*} \label{e:DecA}
a_{ij} \leq K e^{-|i-j|^\gamma}.
\end{equation*}
\item For all $i \in \Z^d$, $J_i: \R \longrightarrow \R$ is a differentiable function such that
\begin{equation*} \label{e:AssJ1}
\frac{d}{dx} J_i(x) \leq 0 \ \ \ \forall \ x \in \R;
\end{equation*}
and for some $\kappa, \kappa^{'}>0$
\begin{equation*} \label{e:AssJ2}
|J_i(x)| \leq \kappa^{'}(|x|^{\kappa}+1) \ \ \ \ \forall \ \ x \in \R.
\end{equation*}
\item $\eta:=\left(\sup_{j \in \Z^d} \sum_{i \in \Z^d} a_{ij}\right) \vee \left(\sup_{i \in \Z^d} \sum_{j \in \Z^d} a_{ij}\right)<\infty, \ \ \ c:=\inf \limits_{i \in \Z^d, y \in \R} \left(-\frac{d}{dy}J_i(y)\right).$
\end{enumerate}
\end{assumption}
\ \\
\indent \emph{Without loss of generality}, we assume that $I_i(0)=0$ for all
 $i \in \Z^d$ and that $K^{'}=0$, $K=1$ and $\gamma=1$ in Assumption \ref{a:IJ} from now on, i.e.
\begin{equation} \label{e:Aij}
a_{ij} \leq  e^{-|i-j|} \ \ \ \forall \ i,j \in \Z^d.
\end{equation}
Without loss of generality, we also assume from now on
\begin{equation} \label{e:J=0}
J_i(0)=0 \ \ \ \forall \ i \in \Z^d.
\end{equation}
\ \\

Let us now list some notations to be frequently used in the
paper, and then give the main results, i.e. Theorems \ref{t:ConDyn} and \ref{t:Erg}.
\begin{itemize}
\item Define $|i-j|=\sum
_{1 \leq k \leq d}|i_k-j_k|$ for any $i,j \in \Z^d$, define $|\Lambda|$ the cardinality of any given finite set $\Lambda \subset \Z^d$.
\item For the national simplicity, we shall write $\partial_i:=\partial_{x_i}$, $\p_{ij}:=\p^2_{x_i x_j}$ and $\partial^{\alpha}_i:=\partial^{\alpha}_{x_i}$. It is easy to see that
$[\p^{\alpha}_i, \p_j]=0$ for all $i,j \in \Z^d$.
\item For any finite sublattice $\Lambda \subset \subset \mathbb{Z}^d$,
let $C_b(\R^{\Lambda},\R)$ be the bounded continuous function space from $\R^{\Lambda}$ to $\R$, denote
$\mathcal{D}=\bigcup_{\Lambda \subset \subset
\mathbb{Z}^d} C_b(\R^{\Lambda},\R)$ and
$$\mathcal{D}^{k}=\{f \in
\mathcal{D}; f \ {\rm has} \  {\rm bounded} \ 0, \cdots,\ kth \ {\rm order} \ {\rm derivatives}\}.$$
\item For any $f \in \mathcal{D}$, denote $\Lambda(f)$ the
localization set of $f$, i.e. $\Lambda(f)$ is the smallest set
$\Lambda \subset \mathbb{Z}^d$ such that $f \in C_b(\R^{\Lambda},\R)$.
\item For any $f \in C_b(\mathbb B,\R)$, define $||f||=\sup_{x \in \mathbb B} |f(x)|$. For any $f \in \mcl D^1$, define $|\nabla f(x)|^2=\sum_{i \in \Z^d} |\p_i f(x)|^2$ and
    $$|||f|||=\sum_{i \in \Z^d} ||\p_i f||.$$
\end{itemize}

\begin{thm} \label{t:ConDyn}
There exists a Markov semigroup $P_t$ on the space $\mcl B_b(\mathbb B, \R)$
generated by the system \eqref{e:IntSys}.
\end{thm}
\begin{thm} \label{t:Erg}
If $c \geq \eta+\delta$ with any $\delta>0$ and $c, \eta$ defined in
(3) of Assumption \ref{a:IJ}, then there exists some probability
measure $\mu$ supported on $\mathbb B$ such that for all $x \in
\mathbb B$,
$$\lim_{t \rightarrow \infty} P^*_t \delta_x = \mu \ \ \ {\rm weakly}.$$
Moreover, for any $x \in \mathbb B$ and $f \in \mcl D^2$, there exists some
$C=C(\Lambda(f),\eta,c,x)>0$  such that we
have
\begin{equation} \label{e:ExpMix}
\left|\int_{\mathbb B} f(y) d P^*_t \delta_x-\mu(f)\right| \leq C
e^{-\frac18 \wedge \frac \delta 2 t} |||f|||.
\end{equation}
\end{thm}
%%%%%%%%%%%%%%%%%%%%%%%%%%%%%%%%%%%%%%%%%%%%%%%%%%%%%%%%%%%%%%%%%%%%%%%%%%%%%%%%%%%%%%%%%%%%%%%%%%%%%%
%%%%%%%%%%%%%%%%%%%%%%%%%%%%%%%%%%%%%%%%%%%%%%%%%%%%%%%%%%%%%%%%%%%%%%%%%%%%%%%%%%%%%%%%%%%%%%%
\ \\
\subsection{Two key estimates}
\label{s:GenOU}
In this subsection, we shall give an estimate for the operator $a$ and $a+\delta$, where $a$ is defined in Assumption \ref{a:IJ} and $\delta$ is the Krockner's function, and also an estimate for a generalized 1 dimensional Ornstein-Uhlenbeck $\alpha$-stable process
governed by \eqref{e:OUAlpGen}.
\\

\subsubsection{Estimates for $a$ and $a+c\delta$} The lemma below will play an important role in several places such as proving \eqref{e:expfinite}. If $(a_{ij})_{i,j \in \Z^d}$ is the transition probability of a random walk on
$\Z^d$, then \eqref{e:EstDel+a} with $c=0$ gives an estimate for the transition probability of the $n$ steps walk.
\begin{lem} \label{l:EstA}
Let $a_{ij}$ be as in Assumption \ref{a:IJ} and satisfy \eqref{e:Aij}. Define
$$[(c\delta+a)^n]_{ij}:=\sum \limits_{i_1, \cdots i_{n-1} \in \Z^d} (c \delta+a)_{ii_1} \cdots (c\delta+a)_{i_{n-1}j}$$
where $c \geq 0$ is some constant and $\delta$ is the Krockner's function, we have
\begin{equation} \label{e:EstDel+a}
[(c\delta+a)^n]_{ij} \leq (c+\eta)^n \sum_{k \geq |j-i|} (2k)^{nd} e^{-k}
\end{equation}
\end{lem}
\begin{rem} \label{r:Aij}
Without the additional assumption \eqref{e:Aij}, one can also have the similar estimates as above, for instance, as
$|i-j| \geq K^{'}$, $(a^n)_{ij} \leq \eta^n \sum_{k \geq |j-i|} (Ck)^{nd} \exp\{-k^{\gamma/2}\}.$
 The $C>0$ is some constant depending on $K,K^{'}$ and $\gamma$, and will not play any essential roles in the later arguments.
\end{rem}
\begin{proof}
Denote the collection of the (n+1)-vortices pathes connecting $i$ and $j$ by $\gamma^n_{i \sim j}$, i.e.
$$\gamma^n_{i \sim j}=\{(\gamma(i))_{i=1}^{n+1}: \ \gamma(1)=i,  \gamma(2) \in \Z^d, \cdots, \gamma(n) \in \Z^d, \gamma(n+1)=j\},$$
for any $\gamma \in \gamma^n_{i \sim j}$, define its length by
$$|\gamma|=\sum_{k=1}^{n} |\gamma(k+1)-\gamma(k)|.$$
We have
\begin{equation} \label{e:PatEst}
\begin{split}
[(a+c\delta)^n]_{ij}&=\sum_{\gamma \in \gamma^n_{i \sim j}}(a+\delta)_{\gamma(1), \gamma(2)} \cdots
(a+c\delta)_{\gamma(n), \gamma(n+1)}  \\
&\leq \sum_{|\gamma|=|i-j|}^{\infty}  (2|\gamma|)^{dn} (c+\eta)^n e^{-|\gamma|}
 \end{split}
\end{equation}
where the inequality is obtained by the following observations:
\begin{itemize}
\item $\min_{\gamma \in \gamma^n_{i \sim j}} |\gamma| \geq |i-j|$.
\item the number of the pathes in $\gamma^n_{i \sim j}$ with length $|\gamma|$ is bounded by
$[(2|\gamma|)^d]^n$
\item  $(a+c\delta)_{\gamma(1), \gamma(2)} \cdots
(a+c\delta)_{\gamma(n), \gamma(n+1)}=\prod \limits_{\{k;\gamma(k+1)=\gamma(k)\}} (a+c\delta)_{\gamma(k), \gamma(k+1)} \times$
$\prod \limits_{\{k;\gamma(k+1) \neq \gamma(k)\}} a_{\gamma(k), \gamma(k+1)} \leq (c+\eta)^n e^{-|\gamma|}.$
\end{itemize}
\end{proof}

\subsubsection{1d generalized Ornstein-Uhlenbeck $\alpha$-stable processes}
Our generalized $\alpha$-stable processes satisfies the following SDE
\begin{equation} \label{e:OUAlpGen}
\begin{cases}
dX(t)=J(X(t))dt+dZ(t) \\
X(0)=x
\end{cases}
\end{equation}
where $X(t),x \in \R$, $J: \R \rightarrow \R$ is differentiable function with polynomial growth, $J(0)=0$ and
$\frac{d}{dx} J(x) \leq 0$, and $Z(t)$ is a one dimensional symmetric $\alpha$-stable process with
$1<\alpha \leq 2$. One
can write $J(x)=\frac{J(x)}{x} x$, clearly $\frac{J(x)}{x} \leq 0$ with the above assumptions
(it is natural to define $\frac{J(0)}{0}=J^{'}(0)$).
$J(x)=-c x \  (c>0)$ is a special case of the above $J$, this is the motivation to call
\eqref{e:OUAlpGen} the generalized Ornstein-Uhlenbeck $\alpha$-stable processes. The following uniform bound is important for proving  (2) of Proposition \ref{p:AppExiDif}. \\
\begin{prop} \label{p:ErgCriGen}
Let $X(t)$ be the dynamics governed by \eqref{e:OUAlpGen} and denote $\mcl E(s,t)=\exp \{\int_s^t \frac{J(X(r))}{X(r)} dr\}$. If $\sup \limits_{x \in \R} \frac{J(x)}{x} \leq -\e$ with any $\e>0$,
then
\begin{equation} \label{e:UniBouGenOU}
\E_x\left|\int_0^t \mcl E (s,t)dZ_s\right|<C(\alpha,\e)
\end{equation}
where $C(\alpha,\e)>0$ only depends on $\alpha,\e$.
In particular, if $J(x)=-\e x$, $X(t)$ is $L^1$ ergodic, i.e. there exists some random variable $\xi \in L^1(\mathbb P)$, which is independent of $x$, such that $X(t) \stackrel{L^1}{\rightarrow} \xi$.
\end{prop}
\begin{proof}
From (1) of Proposition \ref{p:AppExiDif}, we have
\begin{equation}
X(t)=\eot x+\int_0^t \est d Z(s).
\end{equation}

By integration by parts formula (\cite{Bi02}),
\begin{equation*}
\begin{split}
& \ \ \E \left|\int_0^t \est d Z(s)\right| \\
&=\E \left|Z(t)-\int_0^t Z(s) d \est\right|\\
& \leq \E \left|Z(t) \eot\right|+ \E \left|\int_0^t \left(Z(t)-Z(s) \right) d\est\right|.
 %& \leq
 %\sup_{0 \leq s \leq t} \frac{|Z(t)-Z(s)|}{|t-s|^{1/\alpha} \vee 1} \int_0^t [(t-s)^{1/\alpha} \vee 1] %\est \left(-\frac{J(X(s))}{X(s)}\right) ds \\
 %& \ \ \ +\sup_{0 \leq s \leq t}\frac{|Z(s)|} \eot
 \end{split}
\end{equation*}
By \eqref{e:AlpStaPro}, the first term on the r.h.s. of the last line is bounded by
$$\E \left|Z(t) \eot\right| \leq e^{-\e t} \E|Z(t)| \leq C e^{-\e t} t^{1/\alpha} \rightarrow 0 \ \ \ (t\rightarrow \infty).$$
As for the second term, one has
\begin{equation*}
\begin{split}
& \ \ \E \left|\int_0^t \frac{\left(Z(t)-Z(s) \right)}{(t-s)^{1/\gamma} \vee 1} \left[(t-s)^{1/\gamma} \vee 1\right] d \est\right| \\ & \leq \E \left(\sup_{0 \leq s \leq t} \left|\frac{\left(Z(t)-Z(s) \right)}{(t-s)^{1/\gamma} \vee 1}\right| \left|\int_0^t \left[(t-s)^{1/\gamma} \vee 1\right] d \est\right|\right) \\
%& \leq \left\{\E \sup_{0 \leq s \leq t} %\left|\frac{Z(t-s)}{(t-s)^{\gamma}}\right|^{\beta}\right\}^{1/\beta} \left\{\E\left|\int_0^t %(t-s)^{1/\gamma} d \est \right|^{\frac{\beta}{\beta-1}}\right\}^{1-\frac{1}{\beta}}
\end{split}
\end{equation*}
where $1<\gamma<\alpha$. It is easy to see that $\frac{d \est}{1-\eot}$ is a probability
measure on $[0,t]$, by Jessen's inequality, we have
\begin{equation*}
\begin{split}
& \ \ \left(\int_0^t (t-s)^{1/\gamma} \vee 1 d \est
\right)\\
&=\left(\int_0^t (t-s) \vee 1
\frac{d \est}{1-\eot}
\right)^{1/\gamma} (1-\eot) \\
&  \leq \left(\int_0^t (t-s) \vee 1 d \est \right)^{1/\gamma} \leq \left(\int_0^t \est ds\right)^{1/\gamma}+t\eot \leq C(\e,\gamma).
\end{split}
\end{equation*}
On the other hand, by Doob's martingale inequality and $\alpha$-stable property
\eqref{e:AlpStaPro}, for all $N \in \N$, we have
\begin{equation*}
\begin{split}
\E \sup_{1 \leq t \leq 2^N} \left|\frac{Z(t)}{t^{1/\gamma}}\right| & \leq \E \sum_{i=1}^{N}
 \sup_{2^{i-1} \leq t \leq 2^i} \left|\frac{Z(t)}{t^{1/\gamma}}\right| \leq \sum_{i=1}^{N}
\frac{\E \sup_{2^{i-1} \leq t \leq 2^i} |Z(t)|}{2^{(i-1)/\gamma}} \\
& \leq C \sum_{i=1}^{N} \frac{2^{i/\alpha}}{2^{(i-1)/\gamma}}  \leq C(\alpha,\gamma).
%& \leq \left\{\E \sup_{0 \leq s \leq t} %\left|\frac{Z(t-s)}{(t-s)^{\gamma}}\right|^{\beta}\right\}^{1/\beta} \left\{\E\left|\int_0^t %(t-s)^{1/\gamma} d \est \right|^{\frac{\beta}{\beta-1}}\right\}^{1-\frac{1}{\beta}}
\end{split}
\end{equation*}
From the above three inequalities, we immediately have
\begin{equation*}
\begin{split}
\E \left|\int_0^t \left(Z(t)-Z(s) \right) d \est\right|  \leq C(\alpha, \gamma, \e).  %& \leq \left\{\E \sup_{0 \leq s \leq t} %\left|\frac{Z(t-s)}{(t-s)^{\gamma}}\right|^{\beta}\right\}^{1/\beta} \left\{\E\left|\int_0^t %(t-s)^{1/\gamma} d \est \right|^{\frac{\beta}{\beta-1}}\right\}^{1-\frac{1}{\beta}}
\end{split}
\end{equation*}
Collecting all the above estimates, we conclude the proof of \eqref{e:UniBouGenOU}.
\\

As $J(x)=-\e x$, it is clear that
$\eot x \rightarrow 0$ as $t \rightarrow 0$. On the other hand,
by \eqref{e:UniBouGenOU}, the easy fact that $\int_0^t \est dZ(s)$ is
a submartingale, and the submartingale convergence theorem, we immediately have that $\int_0^t \est dZ(s)$ converges to some
random variable $\xi$ in $L^1$ sense as $t \rightarrow \infty$. It is easy to see that $\xi$ is independent of the initial data $x$, thus $X(t)$ is $L^1$ ergodic.
\end{proof}
\section{Existence of Infinite Dimensional Interacting $\alpha$-stable Systems}
In order to prove the existence theorem of the equation \eqref{e:IntSys}, we shall first study
its Galerkin approximation, and uniformly bound some approximate quantities. To pass to the Galerkin approximation limit, we need
to apply a well known estimate in interacting particle systems -- finite speed of propagation of information property.

\subsection{Galerkin Approximation}
 Denote $\Gamma_N:=[-N,N]^d$, which is a cube in ${\mathbb Z}^d$ centered
 at origin. We approximate the infinite dimensional system by
\begin{equation} \label{e:GalApp}
\begin{cases}
d X^N_{i}(t)=[J_i(X_i^N(t))+I^N_i(X^N(t))]dt+dZ_i(t), \\
X^N_{i}(0)=x_i,
\end{cases}
\end{equation}
for all $i \in \Gamma_N$,
where $x^N=(x_i)_{i \in \Gamma_N}$ and $I^N_{i}(x^N)=I_i(x^N,0)$.
It is easy to see that \eqref{e:GalApp} can be written in the following
vector form
\begin{equation} \label{e:AppEqu}
\begin{cases}
dX^N(t)=[J^N (X^N(t))+I^N(X^N(t))]dt+dZ^N(t), \\
X^N(0)=x^N
\end{cases}
\end{equation}
The infinitesimal generator of \eqref{e:AppEqu} (\cite{ARW00}, \cite{XZ09}) is
\begin{align*}
\mathcal{L}_N&=\sum_{i \in \Gamma_N} \p^{\alpha}_i+
\sum_{i \in \Gamma_N} \left[J_i(x^N_i)+I^N_i(x^N)\right] \p_i,
\end{align*}
it is easy to see that
\begin{equation} \label{e:ComLNPi}
[\p_k, \mcl L_N]=\left(\p_k J_k(x^N_k)\right) \p_k+
\sum_{i \in \Gamma_N} \left(\p_k I^N_i(x^N)\right) \p_i.
\end{equation}
\ \\
\indent The following proposition is important for proving the main theorems. (3) is the key estimates for obtaining the limiting semigroup of \eqref{e:IntSys}, while (2) plays the crucial role in proving the ergodicity.
\begin{prop} \label{p:AppExiDif} Let $I_i, J_i$ satisfy Assumption \ref{a:IJ}, together with \eqref{e:Aij} and \eqref{e:J=0}, then
\begin{enumerate}
\item \eqref{e:AppEqu} has a unique mild solution $X^N(t)$ in the sense that for each $i \in \Gamma_N$,
\begin{equation*}
X_i(t)=\mcl E_i(0,t) x_i+\int_0^t
\mcl E_i(s,t) I^N_i(X^N(s)) ds+\int_0^t \mcl E_i(s,t) d Z_i(s),
\end{equation*}
where $\mcl E_i(s,t)=\exp \{\int_s^t \frac{J_i(X^N_i(r))}{X^N_i(r)}dr\}$ with $\frac{J_i(0)}{0}:=J^{'}_i(0)$.
\item For all $x \in B_{R, \rho}$, if $c>\eta$ with $c,\eta$ defined in
(3) of Assumption \ref{a:IJ}, we have
$$\E_x [|X^N_{i}(t)|] \leq C(\rho,R,d,\eta,c) (1+|i|^\rho).$$
\item For all $x \in B_{R, \rho}$, we have
$$\E_x[|X^N_{i}(t)|] \leq C(\rho,R,d)(1+|i|^{\rho})(1+t)e^{(1+\eta)t}.$$
\item For any $f \in C_b^2(\R^{\Gamma_N},\R)$, define $P^N_t f(x)=\E_x[f(X^N(t))],$
we have $P_t^N f(x) \in C^2_b(\R^{\Gamma_N},\R)$.
\end{enumerate}
\end{prop}
\begin{proof}
To show (1), we first formally write down the mild solution as in (1), then apply the classical Picard iteration (\cite{Bi02}, Section 5.3). We can also prove (1) by the method as in the appendix. \\

For the notational simplicity,
we shall drop the index $N$ of the quantities if no confusions arise.
By (1), we have
\begin{equation} \label{e:XiMilSol}
X_i(t)=\mcl E_i(0,t) x_i+\int_0^t
\mcl E_i(s,t) I_i(X^N(s)) ds+\int_0^t \mcl E_i(s,t) d Z_i(s).
\end{equation}
By (1) of Assumption \ref{a:IJ} (w.l.o.g. we assume $I_i(0)=0$ for all $i$),
\begin{equation} \label{e:AbsValXi}
\begin{split}
|X_i(t)|
& \leq \sum_{j \in \Gamma_N} \delta_{ji}\left(|x_j|+\left|\int_0^t
\mcl E_j(s,t) d Z_j(s)\right|\right)\\
& \ \ \ \ +\int_0^t e^{-c(t-s)} \sum_{j \in \Gamma_N} a_{ji} |X_j(s)| ds.
\end{split}
\end{equation}
We shall iterate the the above inequality in two ways, i.e. the following Way 1 and Way 2, which are the methods to show (2) and (3) respectively.
 The first way is under the condition $c>\eta$, which is crucial for obtaining a upper bound of $\E |X_i(t)|$ uniformly for $t \in [0,\infty)$, while the second one is without any restriction, i.e. $c \geq 0$, but one has to pay a price of an exponential growth in $t$.   \\

\noindent \underline{Way 1: The case of $c>\eta$.} By the definition of $c, \eta$ in (3) of Assumption \ref{a:IJ}, \eqref{e:AbsValXi} and Proposition \ref{p:ErgCriGen},
\begin{equation} \label{e:ExpXi}
\begin{split}
\E |X_i(t)|
& \leq \sum_{j \in \Z^d} \delta_{ji} (|x_j|+C(c))+\int_0^t e^{-c(t-s)}\sum_{j \in \Z^d} a_{ji} \E |X_j(s)| ds.
\end{split}
\end{equation}
Iterating \eqref{e:ExpXi} once, one has
\begin{equation} \label{e:IteTwo}
\begin{split}
 \E |X_i(t)| & \leq \sum_{j \in \Z^d} \delta_{ji} (|x_j|+C(c))+\sum_{j \in \Z^d} \frac{a_{ji}}{c} (|x_j|+C(c))\\
& \ \ +\int_0^t e^{-c(t-s)} \int_0^s e^{-c(s-r)} \sum_{j \in \Z^d} (a^2)_{ji} \E |X_j(r)| drds,
\end{split}
\end{equation}
where $C(c)>0$ is some constant only depending on $c$ and $\alpha$ (but we omit $\alpha$ since it does not play any crucial role here).
Iterating \eqref{e:ExpXi} infinitely many times, we have
\begin{equation} \label{e:IteEXi}
\begin{split}
\E |X_i(t)| & \leq \sum_{n=0}^M \frac{1}{c^n} \sum_{j \in \Z^d} (a^n)_{ji} (|x_j|+C(c))+R_M\\
& \leq \sum_{n=0}^\infty \frac{1}{c^n} \sum_{j \in \Z^d} (a^n)_{ji} |x_j|+\frac{C(c)}{1-\eta/c}
%& \leq \sum_{n=0}^\infty \frac{t^n}{n!} \sum_{j \in \Z^d} [(\delta+a)^n]_{ij} (|x_j|+Ct^{\frac 1 %\alpha}) \\
%& \leq \sum_{n=0}^\infty \frac{t^n}{n!} \sum_{j \in \Z^d} [(\delta+a)^n]_{ij} |x_j|+C t^{\frac 1 %\alpha} e^{t(1+\eta)}
\end{split}
\end{equation}
where $R_M$ is an $M$-tuple integral (see the double integral in \eqref{e:IteTwo}) and $\lim_{M \rightarrow \infty} R_M=0$.
 To estimate the double summation in the last line, we split the sum '$\sum_{j \in \Z^d} \cdots$' into two pieces, and control them by \eqref{e:EstDel+a} and $\frac{1}{c^n}$ respectively. More precisely, let $\Lambda(i,n) \subset \Z^d$ be a cube centered at $i$ such that $dist(i,\Lambda^c(i,n))=n^2$ (up to some $O(1)$ correction), one has
\begin{equation} \label{e:SumBou}
\begin{split}
 \sum_{n=1}^{\infty} \frac{1}{c^n} \sum_{j \in \Z^d} (a^n)_{ji} |x_j|=\sum_{n=1}^{\infty} \frac{1}{c^n} \left(\sum_{j \in \Lambda(i,n)}+\sum_{j \in \Lambda^{c}(i,n)}\right)(a^n)_{ji} |x_j|.
\end{split}
\end{equation}
Since $x \in B_{R, \rho}$, we have by \eqref{e:EstDel+a} with $c=0$ therein
\begin{equation} \label{e:OutLam}
\begin{split}
& \ \ \sum_{n=0}^{\infty}\frac{1}{c^n} \sum_{j \in \Lambda^{c}(i,n)} (a^n)_{ji} |x_j| \\
& \leq R \sum_{n=0}^{\infty}\frac{1}{c^n} \sum_{j \in \Lambda^{c}(i,n)} (a^n)_{ji} (|j|^\rho+1) \\
 & \leq C(R,\rho) \sum_{n=0}^{\infty}\frac{1}{c^n} \sum_{j \in \Lambda^{c}(i,n)} (a^n)_{ji}  (|j-i|^\rho+|i|^\rho+1) \\
& \leq C(R,\rho) \sum_{n=0}^{\infty} \frac{\eta^n}{c^n} \sum_{j \in \Lambda^{c}(i,n)}  \sum_{k \geq |j-i|} (2k)^{nd} e^{-\frac12 k} e^{-\frac12 k} \left(|j-i|^{\rho}+|i|^{\rho}+1\right) \\
& \leq C(R,\rho) \sum_{n=1}^{\infty} \frac{\eta^n}{c^n} \sum_{k \geq n^2} (2k)^{nd} e^{-\frac12 k} \sum_{j \in \Lambda^{c}(i,n)} e^{-\frac12 |j-i|} \left(|j-i|^{\rho}+|i|^{\rho}+1\right) \\
& \leq C(\rho, R, d)(1+|i|^{\rho})
\end{split}
\end{equation}
where the last inequality is by the fact $\sum_{k \geq n^2} (2k)^{nd} e^{-\frac12 k} \leq \sum_{k \geq 1} e^{-\frac12 k+nd \log (2k)}<\infty$ and the fact $\sum_{j \in \Lambda^{c}(i,n)} e^{-\frac12 |j-i|} |j-i|^\rho \leq \sum_{j \in \Z^d} e^{-\frac12 |j-i|} |j-i|^\rho<\infty$. For the other piece, one has
\begin{equation} \label{e:InLam}
\begin{split}
& \ \ \sum_{n=0}^{\infty} \frac{1}{c^n} \sum_{j \in \Lambda(i,n)} (a^n)_{ji}|x_j|  \\
& \leq C(R,\rho) \sum_{n=0}^{\infty} \frac{1}{c^n} \sum_{j \in \Lambda(i,n)} (a^n)_{ji}\left( |j-i|^\rho+|i|^\rho+1 \right) \\
& \leq C(R,\rho) \sum_{n=0}^{\infty} \frac{\eta^n}{c^n} |\Lambda(i,n)| \left(n^{2\rho}+|i|^{\rho}+1\right) \\
& \leq C(\rho, R) \sum_{n=0}^{\infty} \frac{\eta^n}{c^n} n^{2d}
\left(n^{2\rho}+|i|^{\rho}+1\right) \\
& \leq C(R, \rho, \eta, c) (1+|i|^\rho). \\
\end{split}
\end{equation}
Collecting \eqref{e:IteEXi}, \eqref{e:OutLam} and \eqref{e:InLam}, we immediately obtain (2).
\\

\noindent  \underline{Way 2: The general case of $c \geq 0$.} By the integration by parts,
Doob's martingale inequality and the easy relation $d \mcl E_j(s,t)=\mcl E_j(s,t) [-L_j(X(s))]ds$ where $L_j(x)=\frac{J_j(x)}{x}$, we have
\begin{equation} \label{e:IntByParJ}
\begin{split}
& \ \E \left |\int_0^t \mcl E_j(s,t) dZ_j(s)\right| \\
& \leq \E|Z_j(t)|+\E \left |\int_0^t \mcl E_j(s,t) L_j(X(s))Z_j(s) ds\right| \\
& \leq C t^{1/\alpha}+\E \left[\sup_{0 \leq s \leq t} |Z_j(s)|
\left|\int_0^t \mcl E_j(s,t) (-L_j(X(s)))ds\right| \right] \\
& \leq Ct^{1/\alpha}+ \E\sup_{0 \leq s \leq t} |Z_j(s)| \\
& \leq C t^{1/\alpha}.
\end{split}
\end{equation}
By \eqref{e:AbsValXi} and \eqref{e:IntByParJ}, one has
\begin{equation} \label{e:Way2}
\begin{split}
\E |X_i(t)|
& \leq \sum_{j \in \Z^d} \delta_{ji} (|x_j|+Ct^{\frac 1 \alpha})+\int_0^t \sum_{j \in \Z^d} (\delta+a)_{ji} \E |X_j(s)| ds
\end{split}
\end{equation}
Iterating the above inequality infinitely many times,
\begin{equation}
\begin{split}
\E |X_i(t)| \leq \sum_{n=0}^\infty \frac{t^n}{n!} \sum_{j \in \Z^d} [(\delta+a)^n]_{ji} |x_j|+Ce^{(1+\eta)t} t^{\frac 1 \alpha},
\end{split}
\end{equation}
By estimating the double summation in the last line by the same method as in Way 1, we finally obtain (3).
\ \\

(4) immediately follows from Proposition 5.6.10 and Corollary 5.6.11 in \cite{Bi02}.
\end{proof}
\subsection{Finite speed of propagation of information property}
The following relation \eqref{e:expfinite} is usually called finite speed of propagation of information property (\cite{GuZe03}),
which roughly means that the effects of the initial condition (i.e. $f$ in our case) need a long time
to be propagated (by interactions) far away. The main reason for this phenomenon is that
the interactions are finite range or sufficiently weak at long range. \\

 From the view point of PDEs, \eqref{e:expfinite} implies equicontinuity of $P^N_tf(x)$ under product topology on any $B_{\rho,R}$, combining this with the fact that $P_t^Nf(x)$ are uniformly bounded, we can find some subsequence $P^{N_k}_t f(x)$ uniformly converge to a limit $P_t f(x)$ on $B_{\rho,R}$ by Ascoli-Arzela Theorem (notice that $B_{\rho,R}$ is compact under product topology). This is also another motivation of establishing the estimates \eqref{e:expfinite}.
\begin{lem}  \label{l:FinSpePro}
\ \\ \
\noindent 1. For any $f \in \D^2$, we have
\begin{equation} \label{e:GraNorEst}
\sum_{k \in \Z^d} ||\p_k P^N_tf||^2 \leq e^{2\eta t} |||f|||^2.
\end{equation}
and
\begin{equation} \label{e:GraNorEst1}
|||P^N_t f||| \leq C(I,t)|||f|||.
\end{equation}
where $C(I,t)>0$, depending on the interaction $I$ and $t$, is an increasing function of $t$. \\
2. \emph{(Finite speed of propagation of information property)} Given any $f \in \D^2$ and $k \notin \Lambda(f)$, for any $0<A \leq 1/4$, there exists
some $B \geq 8$ such that when $n_k>Bt$, we have
\begin{equation} \label{e:expfinite}
||\partial_k P^N_t f||^2\leq 2e^{-At-An_k}|||f|||^2
\end{equation}
where $n_k=[\sqrt{dist(k,\Lambda(f))}]$.
\end{lem}
\begin{proof}
For the notational simplicity, we shall drop the parameter $N$ of $P^N_t$ in the proof.
By the fact $\lim_{t \rightarrow 0+}\frac{P_tF^{2}-F^{2}}{t} \geq \lim_{t \rightarrow 0+}\frac{(P_tF)^{2}-F^{2}}{t}$, one has $\mathcal{L}_N F^{2}-2 F
\mathcal{L}_N F \geq 0.$ Hence,
for any $f \in \D^2$, by \eqref{e:ComLNPi} and the fact $\p_k J_k \leq 0$, we have the
following calculation
\begin{equation} \label{e:ForBacSem}
\begin{split}
\frac{d}{ds}P_{t-s} (\p_k P_{s}f)^{2}&=-P_{t-s} \left[\mathcal{L}_N (\p_k
P_{s} f)^{2}-2  (\p_k P_{s} f)
 \p_k (\mathcal{L}_N P_{s} f) \right] \\
&=-P_{t-s}\left[\mathcal{L}_N (\p_k
P_{s} f)^{2}-2(\p_k P_{s} f)
 \mathcal{L}_N (\p_k P_{s} f) \right] \\
&\ \ \ +2 P_{t-s} \left((\p_k P_{s} f)
[\p_k,\mathcal{L}_N] P_{s} f \right) \\
& \leq 2P_{t-s}\left((\p_k P_{s} f)
[\p_k,\mathcal{L}_N] P_{s} f \right) \\
&=2P_{t-s}\left((\p_k P_{s} f)
\sum_{i \in \Gamma_N} (\p_k I_i) \p_i P_{s} f \right) \\
& \ \ +2P_{t-s}\left((\p_k P_{s} f) (\p_k J_k) \p_k P_{s} f \right) \\
& \leq 2 P_{t-s}\left((\p_k P_{s} f)
\sum_{i \in \Gamma_N} (\p_k I_i) \p_i P_{s} f \right). \\
\end{split}
\end{equation}
Moreover, by the above inequality, Assumption \ref{a:IJ}, and the inequality of arithmetic and geometric means in order,
\begin{equation*}
\begin{split}
|\p_k P_{t} f|^{2} & \leq ||\p_k f||^{2}+2 \int_0^t P_{t-s} \left(|\p_k P_{s} f|
\sum_{i \in \Gamma_N} |\p_k I_i| |\p_i P_{s} f| \right) ds \\
& \leq ||\p_k f||^{2}+\eta \int_0^t P_{t-s} (|\p_k P_{s} f|^{2})ds+\int_0^t P_{t-s} \left(
\sum_{i \in \Gamma_N}  a_{ki} |\p_i P_{s} f|^2 \right) ds \\
&\leq ||\p_k f||^{2}+ \int_0^t P_{t-s} \left(\sum_{i \in \Z^d} (a_{ki} +\eta \delta_{ki}) |\p_i P_sf|^2 \right)ds.
\end{split}
\end{equation*}
where $\eta$ is defined in (3) of Assumption \ref{a:IJ}.
Iterating the above inequality, we have
\begin{equation*}
\begin{split}
|\p_k P_{t} f|^{2}
& \leq ||\p_k f||^{2}+t \sum_{i \in \Z^d} (a_{ki}+\eta \delta_{ki}) ||\p_i f||^2 \\
& \ \ +\int_0^t P_{t-s_1}\int_0^{s_1} P_{s_1-s_2} \sum_{i \in \Z^d} [(a+\eta \delta)^2]_{ki} |\p_i P_{s_2} f|^2 ds_2 ds_1 \\
& \leq \cdots \cdots  \leq \sum_{n=0}^{N} \frac{t^n}{n!} \sum_{i \in \Z^d} [(a+\eta \delta)^n]_{ki} ||\p_i f||^2+Re(N) \\
\end{split}
\end{equation*}
where $Re(N) \rightarrow 0$ as $N \rightarrow \infty$. Hence,
\begin{equation} \label{e:PkPtfEst}
||\p_k P_{t} f||^{2} \leq \sum_{n=0}^{\infty} \frac{t^n}{n!} \sum_{i \in \Z^d} [(a+\eta \delta)^n]_{ki} ||\p_i f||^2.
\end{equation}
Summing $k$ over $\Z^d$ in the above inequality, one has
\begin{equation*}
\begin{split}
\sum_{k \in \Z^d} ||\p_k P_t f||^2 & \leq \sum_{k \in \Z^d} \sum_{n=0}^{\infty} \frac{t^n}{n!} \sum_{i \in \Z^d} [(a+\eta \delta)^n]_{ki} ||\p_i f||^2 \\
&\leq \sum_{n=0}^{\infty} \frac{t^n}{n!} \sup_i \sum_{k \in \Z^d} [(a+\eta \delta)^n]_{ki} \sum_{i \in \Z^d} ||\p_i f||^2 \\
& \leq e^{2\eta t} \sum_{i \in \Z^d} ||\p_i f||^2 \leq e^{2\eta t} |||f|||^2
\end{split}
\end{equation*}
As for \eqref{e:GraNorEst1}, one can also easily obtain from \eqref{e:PkPtfEst} that
$\sum_{k \in \Z^d} ||\p_k P^N_t f|| \leq C(I,t) \sqrt{\sum_{i \in \Z^d} ||\p_k f||^2}
\leq C(I,t) |||f|||$ and that $C(I,t)>0$ is an increasing function related to $t$.
\  \\

In order to prove 2, one needs to estimate the double sum of \eqref{e:PkPtfEst} in a more delicate way. We shall split the sum '$\sum_{n=0}^{\infty}$'
into two pieces '$\sum_{n=0}^{n_k}$' and '$\sum_{n=n_k}^{\infty}$' with $n_k=[\sqrt{dist(k,\Lambda(f))}]$, and control them
by \eqref{e:EstDel+a} and some basic calculation respectively. More precisely, for the piece
'$\sum_{n=0}^{n_k}$', by \eqref{e:EstDel+a} and the definition of $n_k=[\sqrt{dist(k,\Lambda(f))}]$, we have
\begin{equation*}
\begin{split}
& \ \ \ \sum_{n=0}^{n_k} \frac{t^n}{n!} \sum_{i \in \Z^d} [(a+\eta \delta)^n]_{ki} ||\p_i f||^2 \\
& \leq \sum_{n=0}^{n_k} \frac{t^n}{n!} \sum_{i \in \Lambda(f)} \sum_{j \geq |k-i|} (2\eta)^n 2^{nd} (j+\Lambda(f))^{dn} e^{-j} ||\p_i f||^2 \\
& \leq e^t \sum_{i \in \Lambda(f)} \sum_{j \geq |k-i|} \exp \left \{dn_k \log[2(2\eta)^{1/d}(j+\Lambda(f))]-\frac14 n^2_k-\frac j4 \right \} e^{-\frac j2 } ||\p_i f||^2 \\
& \leq C(d, \Lambda(f), \eta) e^t \sum_{i \in \Lambda(f)} \sum_{j \geq n^2_k} e^{-\frac j2} ||\p_i f||^2 \\
& \leq C(d,\Lambda(f),\eta) e^t e^{-\frac12 n^2_k} |||f|||^2.
\end{split}
\end{equation*}
For the other piece, it is easy to see
\begin{equation*}
\begin{split}
& \ \ \sum_{n \geq n_k}\frac{t^n}{n!} \sum_{i \in \Z^d} [(a+\eta \delta)^n]_{ki} ||\p_i f||^2 \\
&=\sum_{n \geq n_k}\frac{t^n}{n!} \sum_{i \in \Lambda(f)} [(a+\eta \delta)^n]_{ki} ||\p_i f||^2
\leq \frac{t^{n_k}}{n_k!} e^{2\eta t} |||f|||^2.
\end{split}
\end{equation*}
Combining \eqref{e:PkPtfEst} and the above two estimates, we immediately have
$$||\p_k P_tf||^2 \leq \{C e^{t} e^{-\frac{1}2 n^2_k}+\frac{t^{n_k}}{n_k!} e^{2\eta t}\} |||f|||^2.$$
For any $A>0$, choosing $B \geq 1$ such
that
$$2-\log B+\log(2\eta)+\frac{2\eta}{B} \leq -2A,$$
as $n>Bt$, one has
\begin{equation*} %\label{e:expfinite1}
\begin{split}
& \ \ \frac{t^{n} (2\eta)^{n}}{n!} e^{2\eta t} \leq \exp\{n\log\frac{2\eta}{B}+2n+(2\eta)\frac{n}{B}\} \\
& \leq
\exp\{-2An\} \leq \exp\{-An-At\}.
\end{split}
\end{equation*}
Now take $0<A \leq 1/4$, $B \geq 8$ and $n$ as the above, we can easily check that
$$e^t e^{-\frac 12 n^2} \leq e^{-\frac14 n^2} e^{-\frac 14 n Bt+t} \leq e^{-An-At}.$$
Replacing $n$ by $n_k$, we conclude the proof of \eqref{e:expfinite}.
\end{proof}
%%%%%%%%%%%%%%%%%%%%%%%%%%%%%%%%%%%%%%%
\subsection{Proof of Theorem \ref{t:ConDyn}}
 As mentioned in the previous subsection, by \eqref{e:expfinite} and the fact that $P_t^Nf(x)$ are uniformly bounded, we can find some subsequence $P^{N_k}_t f(x)$ uniformly converges to a limit $P_t f(x)$ on $B_{\rho,R}$ by Ascoli-Arzela Theorem. However, this method cannot give more detailed description of $P_t$ such as Markov property. Hence, we need to analyze $P^N_t f$ in a more delicate way.

\begin{proof} [{Proof of Theorem \ref{t:ConDyn}}]
 We shall prove the theorem by the following two steps:
\begin{enumerate}
\item $P_t f(x):=\lim \limits_{N \rightarrow \infty} P^N_tf(x)$ exists pointwisely on $x \in \mathbb B$ for any $f \in \mcl D^2$ and $t>0$.
\item Extending the domain of $P_t$ to $\mathcal B_b(\mathbb B)$ and proving that $P_t$ is Markov on $\mcl B_b(\mathbb B)$.
\end{enumerate}
\ \\
\noindent \underline{\emph{Step 1}:} To prove (1), it suffices to show that $\{P^N_tf(x)\}_N$ is a cauchy sequence for $x \in B_{R,\rho}$ with any fixed $R$ and $\rho$.
 \ \\

Given any $M>N$ such that $\Gamma_M \supset \Gamma_N \supset \Lambda(f)$, we have by a similar calculus as in \eqref{e:ForBacSem}
\begin{equation*} \label{e:ForBacDif}
\begin{split}
& \ \ \frac{d}{ds}P^M_{t-s} \left(P^M_{s}f-P^N_{s}f\right)^2
\\
&=-P^M_{t-s}\left[\mcl L_M \left(P^M_{s}f-P^N_{s}f\right)^2-2\left(P^M_{s}f-P^N_{s}f\right) \mcl L_M
\left(P^M_{s}f-P^N_{s}f\right)\right]  \\
& \ \ +2P^M_{t-s}\left[\left(P^M_{s}f-P^N_{s}f\right)(\mcl L_M-\mcl L_N)P^N_{s}f\right] \\
& \leq 2P^M_{t-s}\left[\left(P^M_{s}f-P^N_{s}f\right)(\mcl L_M-\mcl L_N)P^N_{s}f\right],
\end{split}
\end{equation*}
moreover, by the facts $\Lambda(P^N_s f)=\Gamma_N$, $\Gamma_M \supset \Gamma_N$
and $\Lambda(J_k)=k$, $$(\mcl L_M-\mcl L_N)P^N_{s}f=\sum_{i \in \Gamma_N} \left(I^M_i(x^M)-I^N_i(x^N)\right) \p_i P^N_sf.$$
Therefore, by Markov property of $P^M_t$, the following easy fact (by fundamental theorem of calculus, definition of $I^M$, and (1) of Assumption \ref{a:IJ})
 $$|I^M(x^M)-I^N(x^N)| \leq \sum_{j \in \Gamma_M \setminus \Gamma_N} a_{ji} |x_j|,$$
the assumption \eqref{e:Aij} (i.e. $a_{ij} \leq e^{-|i-j|}$), and (3) of Proposition \ref{p:AppExiDif} in order, we have for any $x \in B_{R,\rho}$,
\begin{equation} \label{e:PMtf-PNtfSqu}
\begin{split}
& \ \ \ \left(P^M_t f(x)-P^N_t f(x)\right)^2 \\
& \leq  2||f||_{\infty} \int_0^t P^M_{t-s} \left(\sum_{ i \in \Gamma_N} \sum_{j \in \Gamma_M \setminus \Gamma_N} a_{ji} |x_j|||\p_i P^N_s f|| \right)(x)ds \\
& \leq 2||f||_{\infty} \sum_{ i \in \Gamma_N} \sum_{j \in \Gamma_M \setminus \Gamma_N} e^{-|i-j|}\int_0^t \E_x[|X^M_j(t-s)|] ||\p_i P^N_s f|| ds \\
& \leq C(t,\rho,R,d) ||f||_{\infty} \sum_{ i \in \Gamma_N} \sum_{j \in \Gamma_M \setminus \Gamma_N} e^{-|i-j|}(|j|^{\rho}+1) \int_0^t ||\p_i P^N_s f|| ds.
\end{split}
\end{equation}
\\
\indent Now let us estimate the double sum in the last line of \eqref{e:PMtf-PNtfSqu}, the idea is to split the first sum '$\sum_{i \in \Gamma_N}$' into two pieces '$\sum_{i \in \Lambda}$' and '$\sum_{\Gamma_N \setminus \Lambda}$', and control them by $e^{-|i-j|}$ and \eqref{e:expfinite} respectively. More precisely, take a cube $\Lambda \supset \Lambda(f)$ (to be determined later)
inside $\Gamma_N$, we have by \eqref{e:GraNorEst1}
\begin{equation*}
\begin{split}
& \ \ \sum_{ i \in \Lambda} \sum_{j \in \Gamma_M \setminus \Gamma_N} e^{-|i-j|}(|j|^{\rho}+1) \int_0^t ||\p_i P^N_s f|| ds \\
& \leq 2^{\rho} \sum_{i \in \Lambda} \sum_{j \in \Gamma_M \setminus \Gamma_N} e^{-|i-j|}(|j-i|^{\rho}+|i|^{\rho}+1) \int_0^t ||\p_i P^N_s f|| ds  \\
& \leq 2^{\rho}  \int_0^t \sum_{i \in \Lambda} ||\p_i P^N_s f|| ds  \sum_{k \geq dist(\Lambda, \Gamma_M \setminus \Gamma_N)} \sum_{j:|j-i|=k} e^{-k}(k^{\rho}+|\Lambda|^{\rho}+1)\\
& \leq 2^{\rho}t C(I,t) \sum_{i \in \Z^d} ||\p_i f|| \sum_{k \geq dist(\Lambda, \Gamma_M \setminus \Gamma_N)} \left(|\Lambda|+k \right)^d e^{-k}(k^{\rho}+|\Lambda|^{\rho}+1) \\
& \leq \epsilon
\end{split}
\end{equation*}
for arbitrary $\epsilon>0$ as long as $\Gamma_N, \Gamma_M$ (which depend on $\Lambda$, the interaction $I$, $t$) are both sufficiently large.
\ \\

For the piece '$\sum_{\Gamma_N \setminus \Lambda}$', one has by \eqref{e:expfinite}
\begin{equation*}
\begin{split}
& \ \ \sum_{ i \in \Gamma_N \setminus \Lambda} \sum_{j \in \Gamma_M \setminus \Gamma_N} e^{-|i-j|}(|j|^{\rho}+1) \int_0^t e^{t-s} ||\p_i P^N_s f|| ds \\
& \leq 2^{\rho} e^t \sum_{i \in \Gamma_N \setminus \Lambda} \sum_{j \in \Gamma_M \setminus \Gamma_N} e^{-|i-j|}(|j-i|^{\rho}+|i|^{\rho}+1) \int_0^t e^{-As-An_i} ds  \\
& \leq C(t,\rho,A)\sum_{i \in \Gamma_N \setminus \Lambda}(1+|i|^{\rho})e^{-A[dist(i, \Lambda(f))]^{1/2}} \\
& \leq \epsilon
\end{split}
\end{equation*}
as we choose $\Lambda$ big enough so that $dist(\Gamma_N \setminus \Lambda, \Lambda(f))$ is sufficiently large. Combing all the above, we immediately conclude step 1. We denote
$$P_t f(x)=\lim_{N \rightarrow \infty} P^N_t f(x).$$
\ \\
\noindent \underline{\emph{Step 2:}} Proving that $P_t$ is a Markov semigroup on $\mcl B_b(\mathbb B)$.
We first extend $P_t$ to be an operator on $\mcl B_b(\mathbb B)$, then prove this new $P_t$ satisfies semigroup and Markov property. \\

It is easy to see from step 1, for any fixed $x \in \mathbb B$, $P_t$ is a linear functional on
$\mcl D^2$. Since $\mathbb B$ is locally compact (under product topology), by Riesz representation
theorem for linear functional (\cite{Fo99}, pp 223), we have a Radon measure on $\mathbb B$, denoted by $P^{*}_t \delta_x$, so that
\begin{equation}
P_t f(x)=P^{*}_t \delta_x (f).
\end{equation}
By (3) of Proposition \ref{p:AppExiDif}, take any $x \in \mathbb B$, it is clear that
the approximate process $X^N(t,x^N) \in \mathbb B$ a.s. for all $t>0$. Hence, for
all $N>0$, we have
$$P^N_t (1_{\mathbb B})(x)=\E[1_{\mathbb B} (X^N(t,x^N))]=1 \ \ \ \forall \ x \in \mathbb B.$$
Let $N \rightarrow \infty$, by step 1 (noticing $1_{\mathbb B} \in \mcl D^2$), we have for all
$x \in \mathbb B$
$$P_t 1_{\mathbb B}(x)=1,$$
which immediately implies that $P^{*}_t \delta_x$ is a probability measure supported on $\mathbb B$.
With the measure $P^{*}_t \delta_x$, one can easily extend the operator $P_t$ from
$\mcl D^2$ to $\mcl B_b(\mathbb B)$ by bounded convergence theorem since $\mcl D^2$ is dense in $\mcl B_b(\mathbb B)$ under product
topology. \\

Now we prove the semigroup property of $P_t$, by bounded convergence theorem and the dense property
of $\mcl D^2$ in $\mcl B_b(\mathbb B)$, it suffices to prove this property
on $\mcl D^2$.  More precisely, for any $f \in \mcl D^2$, we shall prove that for all $x \in \mathbb B$
\begin{equation} \label{e:SemPro}
P_{t_2+t_1} f(x)=P_{t_2}P_{t_1} f(x).
\end{equation}
To this end, it suffices to show \eqref{e:SemPro} for all $x \in B_{R,\rho}$.
\ \\

On the one hand, from the first step, one has
\begin{equation} \label{e:LimPNt2t1Pt2t1}
\lim_{N \rightarrow \infty} P^N_{t_2+t_1} f(x)=P_{t_2+t_1}f(x) \ \ \ \forall \ x \in B_{R,\rho}.
\end{equation}
On the other hand, we have
\begin{equation} \label{e:Pt2Pt1-PNt2PNt1}
\begin{split}
|P_{t_2}P_{t_1}& f(x)-P^N_{t_2}P^N_{t_1} f(x)| \leq |P_{t_2}P_{t_1}f(x)-P_{t_2}P^N_{t_1} f(x)| \\
&+|P^M_{t_2} P^N_{t_1} f(x)-P_{t_2}P^N_{t_1} f(x)|+|P^M_{t_2} P^N_{t_1} f(x)-P^N_{t_2}P^N_{t_1}f(x)|,
\end{split}
\end{equation}
with $M>N$ to be determined later according to $N$.
It is easy to have by step 1 and bounded convergence theorem
\begin{equation} \label{e:Pt2Pt1-Pt2PNt1}
|P_{t_2}P_{t_1}f(x)-P_{t_2}P^N_{t_1} f(x)|=|P_{t_2}^* \delta_x(P_{t_1}f-P^N_{t_1}f)| \rightarrow 0
\end{equation}
as $N \rightarrow \infty$. Moreover, by the first step, one has
\begin{equation} \label{e:LimPMt2Pt2}
|P^M_{t_2} P^N_{t_1} f(x)-P_{t_2} P^N_{t_1}f(x)|< \e
\end{equation}
for arbitrary $\e>0$ as long as $M \in \N$ (depending on $\Lambda_N$) is sufficiently large.
As for the last term on the r.h.s. of \eqref{e:Pt2Pt1-PNt2PNt1}, by the
same arguments as in \eqref{e:PMtf-PNtfSqu} and those immediately
after \eqref{e:PMtf-PNtfSqu}, we have
\begin{equation} \label{e:PMt2Pt1-PNt2PNt1}
\begin{split}
& \ \ \ \left(P^M_{t_2}P^N_{t_1} f(x)-P^N_{t_2}P^N_{t_1} f(x)\right)^2 \\
& \leq C(t_1,t_2,\rho,R,d) ||f||_{\infty} \sum_{ i \in \Gamma_N} \sum_{j \in \Gamma_M \setminus \Gamma_N} e^{-|i-j|}(|j|^{\rho}+1) \int_0^{t_2} ||\p_i P^N_{t_1+s} f|| ds \\
&<\epsilon
\end{split}
\end{equation}
for arbitrary $\e>0$ if $\Gamma_M$ and $\Gamma_N$ are both sufficiently large.
\ \\

Collecting \eqref{e:Pt2Pt1-PNt2PNt1}-\eqref{e:PMt2Pt1-PNt2PNt1}, we have
$$\lim_{N \rightarrow \infty} P^N_{t_2} P^N_{t_1} f(x)=P_{t_2}P_{t_1} f(x), $$
which, with \eqref{e:LimPNt2t1Pt2t1} and the fact $P^N_{t_2+t_1}=P^N_{t_2}P^N_{t_1}$, implies  \eqref{e:SemPro} for $x \in B_{R, \rho}$.
\ \\

Since $P_t({\bf 1} )=1$ and $P_t(f) \geq 0$ for any $f \geq 0$, $P_t$ is a Markov semigroup (\cite{GuZe03}).
\end{proof}
%%%%%%%%%%%%%%%%%%%%%%%%%%%%%%%%%%%%%%%%%%%
%%%%%%%%%%%%%%%%%%%%%%%%%%%%%%%%%%%%%%%%%%%%%%%%%%
\section{Proof of Ergodicity Result} \label{s:ErgThe}
%%%%%%%%%%%
The main ingredient of the proof follows
the spirit of Bakry-Emery criterion for logarithmic Sobolev inequality (\cite{BaEm85}, \cite{GuZe03}). In
\cite{BaEm85}, the authors first studied the logarithmic Sobolev inequalities of some diffusion generator by differentiating its first order square field $\Gamma_1(\cdot)$ (see the definition of $\Gamma_1$ and $\Gamma_2$ in chapter 4 of \cite{GuZe03}) and obtained the following relations
\begin{equation} \label{e:BECri}
\frac d{dt} P_{t-s} \Gamma_1(P_sf) \leq -c P_{t-s}\Gamma_2 (P_s f)
\end{equation}
where $P_t$ is the semigroup generated by the diffusion generator, and $\Gamma_2(\cdot)$ is the second order square field. If $\Gamma_2(\cdot) \geq C\Gamma_1(\cdot)$, then one can obtain logarithmic Sobolev inequality. The relation $\Gamma_2(\cdot) \geq C\Gamma_1(\cdot)$ is called \emph{Bakry-Emery criterion}.
\ \\

In our case, one can also compute $\Gamma_1(\cdot), \Gamma_2(\cdot)$ of $P^N_t$, which have the similar relation as \eqref{e:BECri}. It is interesting to apply this relation to prove some regularity
of the semigroup $P^N_t$, but seems hard to obtain the gradient bounds by it. Alternatively, we replace
$\Gamma_1(f)$ by $|\nabla f|^2$, which
is actually not the first order square field of our case but the one of the diffusion generators, and differentiate $P_{t-s} |\nabla P_sf|^2$. We shall see that the following relation \eqref{e:BECri2} plays the same role as the Bakry-Emery criterion.
\begin{lem} \label{l:GraDec}
If $c \geq \eta+\delta$ with any $\delta>0$ and $c, \eta$ defined in (3) of Assumption \ref{a:IJ}, we have
\begin{equation} \label{e:GB}
|\nabla P^N_t f|^{2} \leq e^{-2\delta t} P^N_t|\nabla f|^{2} \ \ \
\forall \ f \in \D^2
\end{equation}
\end{lem}
\begin{proof}
For the notational simplicity, we drop the index $N$ of the quantities.
By a similar calculus as in \eqref{e:ForBacSem}, we have
\begin{equation} \label{e:DifPs2m}
\begin{split}
\frac{d}{ds}P_{t-s} |\nabla P_{s}f|^{2}&=-P_{t-s}\left(\mathcal{L}_N|\nabla
P_{s} f|^{2}-2 \nabla P_{s} f
\cdot \mathcal{L}_N \nabla P_{s}
f \right) \\
&\ \ \ +2P_{t-s}\left(\nabla P_{s} f \cdot
[\nabla,\mcl L_N] P_{s} f \right) \\
& \leq 2P_{t-s} \left(\nabla P_{s} f \cdot
[\nabla, \mcl L_N] P_{s} f \right)  \\
&=2P_{t-s}\left(\sum_{i,j \in \Gamma_N} \p_j I_i(x) \p_i P_s f \p_j P_s f \right) \\
& \ \ +2P_{t-s}\left(\sum_{i \in \Gamma_N} \p_i J_i(x_i) (\p_i P_s f)^2 \right),
\end{split}
\end{equation}
where '$\cdot$' is the inner product of vectors in $\R^{\Gamma_N}$.
Denote the quadratic form by
$$Q(\xi,\xi)=\sum_{i,j \in \Gamma_N} \left[\p_i J_i(x_i) \delta_{ij}+\p_j I_i(x)\right] \xi_i\xi_j \ \  \ \forall \ \xi \in \R^{\Gamma_N},$$
it is easy to see by the assumption that
\begin{equation} \label{e:BECri2}
-Q(\xi,\xi) \geq \delta |\xi|^2.
\end{equation}
This, combining with \eqref{e:DifPs2m}, immediately implies
\begin{equation}
\frac{d}{ds}P_{t-s} |\nabla P_{s}f|^{2} \leq -2 \delta P_{t-s}\left(|\nabla P_{s}f|^{2}\right),
\end{equation}
from which we conclude the proof.
\end{proof}
Let us now combining Lemma \ref{l:GraDec} and the finite speed of propagation of information property \eqref{e:expfinite} to prove the ergodic result.

\begin{proof} [{\bf Proof of Theorem \ref{t:Erg}}]
 We split the proof into the following three steps: \\

\noindent \underline{\emph{Step 1:}} For all $f \in \mcl D^2$, $\lim \limits_{t \rightarrow \infty} P_tf(0)=\ell (f)$ where $\ell(f)$ is some constant depending on $f$. \\

For any $\forall t_2>t_1>0$, we have by triangle inequality
\begin{equation*} \label{e:01}
\begin{split}
|P_{t_2}f(0)-P_{t_1}f(0)| &\leq |P_{t_2}f(0)-P^N_{t_2}f(0)|
+|P^N_{t_2}f(0)-P^N_{t_1}f(0)| \\
& \ \ +|P_{t_1}f(0)-P^N_{t_1}f(0)|.
\end{split}
\end{equation*}
By Theorem \ref{t:ConDyn}, there
exists some $N(t_1,t_2) \in {\mathbb N}$ such that as $N>N(t_1,t_2)$
\begin{equation} \label{e:PtPNt}
|P_{t_2}f(0)-P^N_{t_2}f(0)|+|P_{t_1}f(0)-P^N_{t_1}f(0)|<e^{-\frac{\delta \wedge A}{2}t_1} |||f|||.
\end{equation}
\ \\
\indent Next, we show that for all $N \in \N$,
\begin{equation} \label{e:PNt20-PNt10}
|P^N_{t_2}f(0)-P^N_{t_1}f(0)| \leq C(A,\delta, \Lambda(f)) e^{-\frac{\delta \wedge A}{2} t_1} |||f|||.
\end{equation}
\ \\

 By the semigroup property of $P^N_t$ and fundamental theorem of calculus, one has
\begin{equation} \label{e:ErgodicAt0}
\begin{split}
|P^N_{t_2}f(0)-P^N_{t_1}f(0)|&=
\left|\E_{0} \left[P^N_{t_1}f(X^N(t_2-t_1))-P^N_{t_1}f(0)\right] \right| \\
&=\left|\int_0^1 \E_{0} \left[\frac d{d\lambda} P^N_{t_1}f(\lambda X^N(t_2-t_1)) \right] d \lambda \right|\\
& \leq \int_0^1  \sum_{i \in \Gamma_N} \E_{0} \left[|\p_i P^N_{t_1}f(\lambda X^N(t_2-t_1))|
|X^N_i(t_2-t_1)| \right] d \lambda.
\end{split}
\end{equation}
To estimate the sum '$\sum_{i \in \Gamma_N}$' in the last line, we split it
into two pieces '$\sum_{i \in \Lambda}$' and '$\sum_{i \in \Gamma_N \setminus \Lambda}$', and
control them by Lemma \ref{l:GraDec} and the finite speed of propagation of information property
in Lemma \ref{l:FinSpePro}. Let us show the more details as follows. \\

Take $0<A \leq 1/4$, and let $B=B(A,\eta) \geq 8$ be chosen as in
Lemma \ref{l:FinSpePro}. We choose a cube $\Lambda \supset
\Lambda(f)$ inside $\Gamma_N$ so that $dist(\Lambda^c,
\Lambda(f))=B^2 t^2_1$ (up to some order $O(1)$ correction). On the
one hand,  by \eqref{e:GB}, we clearly have $||\p_i P_tf|| \leq
e^{-\delta t} |||f|||$ for all $i \in \Gamma_N$. Therefore, by (2)
of Proposition \ref{p:AppExiDif},
\begin{equation} \label{e:Erg01}
\begin{split}
& \ \  \sum_{i \in \Lambda} \E_{0} \left[|\p_i P^N_{t_1}f(\lambda X^N(t_2-t_1))|
|X^N_i(t_2-t_1)| \right]\\
 & \leq  \sum_{i \in \Lambda} ||\p_i P^N_{t_1}f|| \E_0 \left[|X^N_i(t_2-t_1)|\right]
 \\
& \leq C \sum_{i \in \Lambda} e^{-\delta t_1}|||f||| (1+|i|^{\rho})
\end{split}
\end{equation}
\ \ \\
\indent As for the piece '$\sum_{i \in \Gamma_N \setminus \Lambda}$', it is clear to see $n_i=\sqrt{dist(i,\Lambda(f))} \geq Bt_1$ for $i \in \Gamma_N \setminus \Lambda$, by Lemma \ref{l:FinSpePro} and (2) of Proposition \ref{p:AppExiDif}, one has
\begin{equation} \label{e:Erg02}
\begin{split}
& \ \  \sum_{i \in \Gamma_N \setminus \Lambda} \E_{0} \left[|\p_i P^N_{t_1}f(\lambda X^N(t_2-t_1))|
|X^N_i(t_2-t_1)| \right]\\
& \leq \sum_{i \in \Gamma_N \setminus \Lambda} ||\p_i P^N_{t_1}f||
\E_{0} \left[|X^N_i(t_2-t_1)| \right] \\
 & \leq   C \sum_{i \in \Gamma_N \setminus \Lambda} e^{-A n_i-At_1} (1+|i|^\rho)
 |||f|||
\end{split}
\end{equation}
Since $0 \in B_{R,\rho}$ with any $R,\rho>0$, we take $\rho=1$ and
$R=1$ in the previous inequalities. Combining \eqref{e:ErgodicAt0},
\eqref{e:Erg01} and \eqref{e:Erg02}, we immediately have
\begin{equation} \label{e:PNt2f0-PNt1f0}
\begin{split}
 & \ \ |P^N_{t_2}f(0)-P^N_{t_1}f(0)| \\
 & \leq C\left[\sum_{i \in \Gamma_N
\setminus \Lambda} e^{-A n_i-\frac{A}{2}t_1}
(1+|i|)+(B^2t^2_1+1+\Lambda(f))^{1+d} e^{-\frac{\delta}{2}
t_1}\right] e^{-\frac{A \wedge \delta}{2}t_1}  |||f|||.
\end{split}
\end{equation}
 and $\sum_{i \in \Gamma_N \setminus \Lambda} e^{-A n_i}
(1+|i|) \leq \sum_{i \in \Z^d \setminus \Lambda} e^{-A n_i}
(1+|i|)<\infty$, whence \eqref{e:PNt20-PNt10} follows. Combining
\eqref{e:PNt2f0-PNt1f0} and \eqref{e:PtPNt}, one has
\begin{equation} \label{e:Pt2f0-Pt1f0}
|P_{t_2}f(0)-P_{t_1}f(0)| \leq C(A,\delta, \Lambda(f)) e^{-\frac{\delta \wedge A}{2} t_1} |||f|||.
\end{equation}
\ \\

\noindent \underline{\emph{Step 2:}}
Proving that $\lim_{t \rightarrow \infty} P_t f(x)=\ell(f)$ for all $x \in \mathbb B$. \\

It suffices to prove that the above limit
is true for every $x$ in one ball $B_{R, \rho}$. By triangle inequality, one has
\begin{equation} \label{e:05}
\begin{split}
|P_tf(x)-\ell(f)| &\leq |P_tf(x)-P^N_tf(x)|+|P^N_tf(x)-P^N_tf(0)|
\\
&\ \ \ +|P^N_tf(0)-P_tf(0)|+|P_tf(0)-\ell(f)|
\end{split}
\end{equation}
By \eqref{e:Pt2f0-Pt1f0},
\begin{equation} \label{e:limit at 0}
|P_tf(0)-\ell(f)|<Ce^{-\frac{A \wedge \delta}{2} t} |||f|||,
\end{equation}
where $C=C(A,\delta,\Lambda(f))>0$. By Theorem \ref{t:ConDyn},
$\forall \ t>0, \ \exists \ N(t,R,\rho) \in {\mathbb N}$ such that
as $N>N(t,R,\rho)$
\begin{equation} \label{e:N approximation at point}
\begin{split}
& |P_tf(x)-P^N_tf(x)|<e^{-\frac{A \wedge \delta}{2}t} |||f|||, \\
& |P^N_tf(0)-P_tf(0)|<e^{-\frac{A \wedge \delta}{2}t} |||f|||.
\end{split}
\end{equation}
\ \

 By an argument similar as in \eqref{e:ErgodicAt0}-\eqref{e:Erg02}, we have
\begin{equation} \label{e:PNtfx-PNtf0}
\begin{split}
& \ \ |P^N_tf(x)-P^N_tf(0)| \leq \sum_{i \in \Z^d} ||\p_i P^N_t f|| |x_i| \\
& \leq  C \left[(B^2t^2_1+1+\Lambda(f))^{\rho+d} e^{-\delta t}
+\sum_{i \in \Gamma_N \setminus \Lambda} e^{-A n_i-At} (1+|i|^\rho) \right] |||f||| \\
& \leq C \left[(B^2t^2+1+\Lambda(f))^{\rho+d} e^{-\frac{\delta}{2}
t} +\sum_{i \in \Gamma_N \setminus \Lambda}
e^{-A n_i-\frac{A}2 t} (1+|i|^\rho) \right] e^{-\frac{A \wedge \delta}{2} t}|||f|||. \\
\end{split}
\end{equation}
Collecting \eqref{e:05}-\eqref{e:PNtfx-PNtf0}, we immediately
conclude Step 2. \\

\noindent \underline{\emph{Step 3:}} Proof of the existence of ergodic
measure $\mu$ and \eqref{e:ExpMix}. \\

 From step 2, for each $f
\in \mcl D^2$, there exists a constant $\ell(f)$ such that
$$\lim_{t \rightarrow \infty} P_tf(x)=\ell(f)$$
for all $x \in \mathbb B$. It is easy to see that $\ell$ is a linear functional on $\mcl D^2$, since
$\mathbb B$ is locally compact (under the product topology), there exists some unsigned Radon measure $\mu$ supported on $\mathbb B$ such that $\mu(f)=\ell (f)$ for all $f \in \mcl D^2$. By the fact that $P_t {\bf 1}(x)=1$ for all $x \in \mathbb B$ and $t>0$, $\mu$ is a probability measure. \\

On the other hand, since $P_t f(x)=P_t^{*} \delta_x(f)$ and $\lim_{t
\rightarrow \infty} P_tf=\mu(f)$, we have $P_t^{*} \delta_x
\rightarrow \mu$ weakly and $\mu$ is strongly mixing. Moreover, by
\eqref{e:05}-\eqref{e:PNtfx-PNtf0}, we immediately have
$$|P_tf(x)-\mu(f)| \leq C(A,\delta,x,\Lambda(f)) e^{-\frac{A \wedge \delta}{2} t}|||f|||,$$
recall that $0<A \leq 1/4$ in 2 of Lemma \ref{l:FinSpePro} and take $A=1/4$ in the above inequality,
we immediately conclude the proof of \eqref{e:ExpMix}.
\end{proof}
\section{Appendix}
In this section, we shall prove (1) of Proposition \ref{p:AppExiDif}, i.e. the existence and uniqueness of strong solutions of \eqref{e:GalApp}. To this end, we first need to introduce Skorohod's topology and a tightness criterion as follows.
\begin{defn} [Skorohod's topology (\cite{An07}, page 29)]
Given any $T>0$, let $D([0,T]; \R^{\Gamma_N})$ be the collection of the functions from
$[0,T]$ to $\R^{\Gamma_N}$ which are \emph{right continuous} and have \emph{left limit}. The Skorohod topology is
given by the following metric $d$
$$d(f,g)=\inf_{\lambda \in \Lambda} \{||f \circ \lambda-g||_{\infty} \vee ||\lambda-e||_\infty\}$$
where $\Lambda$ is the set of the strictly increasing functions mapping $[0,T]$ onto itself such that
both $\lambda$ and its inverse are continuous, and $e$ is the identity map on $[0,T]$.
\end{defn}
In order to prove the tightness of probability measures on $D([0,T];\R^{\Gamma_N})$, we define
$$v_f(t,\delta)=\sup\{|f(t_1)-f(t_2)|; t_1,t_2 \in [0,T] \cap (t-\delta,t+\delta)\},$$
$$w_f(\delta)=\sup\{\min(|f(t)-f(t_1)|,|f(t_2)-f(t)|); t_1\leq t \leq t_2 \leq T,t_2-t_1 \leq \delta\}.$$
\ \\
\indent The following theorem can be found in \cite{An07} (page 29) or \cite{Be98}. Roughly speaking,  the statement (1) below means that
most of the paths are uniformly bounded, while (2) rules out the paths which have large oscillation in
a short time interval.
\begin{thm} \label{t:TigCri}
The sequence of probability measures $\{P_n\}$
is tight in the above Skorohod's topology if
\begin{enumerate}
\item For each $\e>0$, there exists $c>0$ such that
$$P_n\{f: ||f||_{\infty}>c\} \leq \e, \ \ \ \forall \ n.$$
\item For each $\e>0$, there exists some $\delta$ with $0<\delta<T$ and some integer $n_0$ such that as $n \geq n_0$
$$P_n\{f; w_f(\delta) \geq \eta\} \leq \e,$$
and
$$P_n\{f; v_f(0,\delta) \geq \eta\} \leq \e, P_n\{f; v_f(T,\delta) \geq \eta\} \leq \e.$$
\end{enumerate}
\end{thm}
\ \\
\begin{proof} [{\bf Proof of (1) of Proposition \ref{p:AppExiDif}}]
For the notational convenience, if no confusion can arise, we shall drop the index $N$ of the quantities and simply write all the equations and estimates in the vector form. To understand the idea, one can take all  vectors as scalars. The $|\cdot|$ means the absolute value of vectors, i.e. for any $x \in \R^{\Gamma_N}$, $|x|=\sum_{i \in \Gamma_N} |x_i|$.
\\

From the above, the equation \eqref{e:AppEqu} can be written in vector form by
\begin{equation} \label{e:SimSDE}
\begin{cases}
dX(t)=J(X(t))dt+I(X(t))dt+dZ(t), \\
X(0)=x.
\end{cases}
\end{equation}
Recall the assumption $J_i(0)=0$ for all $i \in \Gamma_N$ in \eqref{e:J=0}, we can rewrite the above equation by
\begin{equation} \label{e:NewForX}
dX(t)=\frac{J(X(t))}{X(t)} X(t) dt+I(X(t))dt+dZ(t).
\end{equation}
where $\frac{J(X(t))}{X(t)}=diag\{\frac{J_i(X_i(t))}{X_i(t)}; i \in \Gamma_N\}$ is a diagonal matrix.
By $\frac{J_i(X_i(t))}{X_i(t)} \leq 0$ for all $i \in \Gamma_N$, the term $\frac{J(X(t))}{X(t)}X(t)dt$ in the above equation will drive $X(t)$ to zero. By the Lipschitz property of $I$, the equation \eqref{e:NewForX}
without $\frac{J(X(t))}{X(t)}X(t)dt$ has a unique solution. Combining these two points together, we
expect that \eqref{e:NewForX} has a unique solution. Let us make the above heuristic observation rigorous as follows.
\\

Define
$X^{(0)}(t)=x$  and, for $n \geq 0$, $X^{(n+1)}$ satisfies the following equation
\begin{equation} \label{e:NSte}
 dX^{(n+1)}(t)=\frac{J(X^{(n)}(t))}{X^{(n)}(t)} X^{(n+1)}(t)dt+I(X^{(n+1)}(t))dt+dZ(t).
\end{equation}
Set
$$\estn=\exp \left\{\int_0^t \frac{J(X^{(n)}(s))}{X^{(n)}(s)} ds \right\}.$$
Thanks to $\frac{J_i(X_i(t))}{X_i(t)} \leq 0$ ($i \in \Gamma_N$), by the classical Picard
iteration, (noticing that the stochastic term in \eqref{e:MilSol}
plays no role in the convergence of the iteration), we  have
\begin{equation} \label{e:MilSol}
\begin{split}
X^{(n+1)}(t)=\estn x+\int_0^t
\estn I(X^{(n+1)}(s)) ds+\int_0^t \estn d Z(s).
\end{split}
\end{equation}
\ \\
\emph{\underline{Step 1}: Existence and Uniqueness under the tightness assumption.} We shall prove that the laws $\{P^{(n)}\}$ of $\left\{(X^{(n)}(t))_{0 \leq t \leq T}\right\}$, which are inductively defined by \eqref{e:MilSol}, are tight under the Skorohod topology on
$D([0,T];\R^{\Gamma_N})$ in step 2. With this tightness, one has some probability measure $P$ on $D([0,T];\R^{\Gamma_N})$ and some subsequence of $\{n\}$, still denoting it by $\{n\}$ for notational simplicity, such that
$$P^{(n)} \rightarrow P \ \ {\rm weakly} \ \ {\rm as \ } n \rightarrow \infty.$$
By Skorohod embedding Theorem (see \cite{IkWa81} for the Brownian motion case and
\cite{Pi07}, \cite{Ro76} for more general processes), we have some probability space $(\Omega, \mathcal F, \mcl F_t, \mathbb P_x)$, together with some random variable sequence $\{X^{(n)}\}$ and $X$, (\emph{note} that the $X^{(n)}$ here are not necessary the same as in \eqref{e:MilSol}), satisfying
\begin{itemize}
\item Under $\mathbb P_x$, $X^{(n)}$ have probability $P^{(n)}$ and $X$ has probability $P$.
\item $X^{(n)} \rightarrow X$ as $n \rightarrow \infty$ under Skorohod's topology.
\end{itemize}
From the first property above, one can see that $X^{(n+1)}$ satisfies \eqref{e:NSte} and \eqref{e:MilSol}. More precisely,
\begin{equation} \label{e:MilSol1}
\begin{split}
X^{(n+1)}(t)&=\estn x+\int_0^t
\estn I(X^{(n+1)}(s)) ds+\int_0^t \estn d Z^{(n+1)}(s) \\
&=\estn x+\int_0^t \estn I(X^{(n+1)}(s))ds \\
& \ \ +Z^{(n+1)}(t)+\int_0^t Z^{(n+1)}(s) \estn \frac{J^{(n)}(X^{(n)}(s))}{X^{(n)}(s)}ds.
\end{split}
\end{equation}
where $Z^{(n+1)}$ is a symmetric $\alpha$-stable process depends on $X^{(n+1)}$. Since,
by Doob's martingale inequality and the $\alpha$-stable property, one has
$$\E_x \sup_{0 \leq s \leq t} |Z^{(n+1)}(s)|<\infty, \ \ \ \E_x |Z^{(n+1)}(s_1)-Z^{(n+1)}(s_2)| \leq |s_1-s_2|^{1/\alpha},$$
by the tightness criterion Theorem \ref{t:TigCri} and Skorohod embedding theorem again, we have some subsequence $\{n_k\}$ of
$\{n\}$ so that $Z^{(n_k)} \rightarrow Z$, where the $Z$ is some $|\Gamma_N|$-dimensional standard symmetric $\alpha$-stable processes. \\

Sending $n_k \rightarrow \infty$, by continuity of $J$ and $I$, $X$ satisfies the equation \eqref{e:MilSol1} with $X^{(n)}$ and $X^{(n+1)}$ therein both replaced by $X$. Hence,  $X$ solves \eqref{e:SimSDE} in the mild solution sense. Since
\eqref{e:SimSDE} is a finite dimensional dynamics, by differentiating $t$ on the both side of
this mild solution, we have that $X(t)$ satisfies \eqref{e:SimSDE}, which is equivalent to
\begin{equation} \label{e:StrSol}
X(t)=x+\int_0^t \left[J(X(s))+I(X(s))\right] ds+Z(t).
\end{equation}
So the equation \eqref{e:SimSDE} at least has a weak solution, i.e. there exists a random variable $X(t)$ and a standard $|\Gamma_N|$-dimensional symmetric $\alpha$-stable process $Z(t)$ on $(\Omega, \mcl F, \mcl F_t, \mathbb P_x)$ satisfying \eqref{e:StrSol}.
\\

Suppose that there exists another weak solution $Y$ on $(\Omega, {\mathcal F}, {\mcl F}_t, \tilde{\mathbb P}_x)$. One can see that $Y(t)-X(t)$ satisfies the following equation
\begin{equation}
\frac d{dt} (X(t)-Y(t))=J(X(t))-J(Y(t))+I(X(t))-I(Y(t))
\end{equation}
with $Y(0)-X(0)=0$. By Assumption \ref{a:IJ},
one has $(J(x)-J(y)) \cdot (x-y) \leq 0$, and thus  from the above differential equation one obtains
$$|X(t)-Y(t)|^2 \leq C(N) \int_0^t |X(s)-Y(s)|^2 ds$$
which immediately implies $X(t)-Y(t)=0$ for all $t>0$.
This pathwise uniqueness implies that $X(t)$ is the unique mild solution of \eqref{e:SimSDE} (Chapter V.3 of \cite{RoWi00}, \cite{BaCh06}).
\\

\noindent \emph{\underline{Step 2}: Tightness of $P^{(n)}$.}  Recall that $P^{(n)}$ be the
probability of $(X^{(n)}(t))_{0 \leq t \leq T}$. In order to prove that $P^{(n)}$ is tight in
$D([0,T]; \R^{\Gamma_N})$, by Theorem \ref{t:TigCri}, it suffices to prove the following two inequalities:
for any $n \in \N$,
\begin{equation} \label{e:UniBouT}
\E \sup_{0 \leq t \leq T} |X^{(n)}(t)| \leq e^{CT}(|x|+C(N) T^{1/\alpha})
\end{equation}
\begin{equation} \label{e:EquCon}
\E[|X^{(n)}(t_1)-X^{(n)}(t_2)|] \leq C(|x|,T, N) |t_1-t_2|^{\delta}  \ \ \ \ \forall \  0 \leq t_1, t_2 \leq T.
\end{equation}
with $\delta=\delta(I,J)>0$.  \\

By \eqref{e:MilSol}, triangle inequality and the Lipschitz condition of $I$ (w.l.o.g. assume $I(0)=0$),
one has
\begin{equation*} %\label{e:UniBou}
\begin{split}
\E \sup_{0 \leq s \leq t} |X^{(n+1)}(s)| & \leq |x|+C \int_0^t \E \sup_{0 \leq r \leq s} |X^{(n+1)}(r)| ds+\E \left |\int_0^t
e^{\int_s^t \frac{J(X^{(n)}(r))}{X^{(n)}(r)} dr} dZ(s)\right|  \\
\end{split}
\end{equation*}
moreover, by the same argument as in \eqref{e:IntByParJ},
\begin{equation} \label{e:IntByPar}
\begin{split}
\E \left |\int_0^t \estn dZ(s)\right| \leq C(N) t^{1/\alpha}.
\end{split}
\end{equation}
Hence,
\begin{equation*} %\label{e:UniBou}
\begin{split}
\E \sup_{0 \leq s \leq t} |X^{(n+1)}(s)|
& \leq |x|+C \int_0^t \E \sup_{0 \leq r \leq s} |X^{(n+1)}(r)| ds+C(N)T^{1/\alpha},
\end{split}
\end{equation*}
which easily implies \eqref{e:UniBouT}.
\ \\

Now we prove \eqref{e:EquCon}. By triangle inequality, we have
\begin{equation} \label{e:Xt2-Xt1}
\begin{split}
|&X^{(n+1)}(t_2)-X^{(n+1)}(t_1)| \leq |(\mcl E^{(n)}(0,t_2)-\mcl E^{(n)}(0,t_1))x| \\
& \ +\left|\int_0^{t_2} \mcl E^{(n)} (s,t_2) I(X^{(n+1)}(s))ds-\int_0^{t_1} \mcl E^{(n)}(s,t_1) I(X^{(n+1)}(s))ds \right| \\
& \ +\left|\int_0^{t_2} \mcl E^{(n)} (s,t_2) dZ(s)-\int_0^{t_1} \mcl E^{(n)}(s,t_1) dZ(s)\right| \\
&=A_1(t)+A_2(t)+A_3(t)
\end{split}
\end{equation}
where $A_1(t),A_2(t),A_3(t)$ denote in order the three terms on the r.h.s. of the inequality, and they
can be estimated by the same argument. We shall show this argument by $A_3$ (which, among the three terms, is the
most difficult one) as follows.
\ \\

By integration by part formula, one has
\begin{equation}
\begin{split}
A_3 & \leq |Z(t_2)-Z(t_1)|+|\int_0^{t_2} Z(s) \mcl E^{(n)} (s,t_2) \frac{J(X^{(n)}(s))}{X^{(n)}(s)} ds\\
& \ \ \ \ -\int_0^{t_1} Z(s) \mcl E^{(n)}(s,t_1) \frac{J(X^{(n)}(s))}{X^{(n)}(s)} ds|
\end{split}
\end{equation}
By the $\alpha$-stable property of $Z(t)$, one has  $\E[|Z(t_2)-Z(t_1)|] \leq C|t_2-t_1|^{1/\alpha}$.
For the second term on the r.h.s. of the inequality, we have
\begin{equation}
\begin{split}
& \ \left|\int_0^{t_2} Z(s) \mcl E^{(n)} (s,t_2) \frac{J(X^{(n)}(s))}{X^{(n)}(s)} ds-\int_0^{t_1} Z(s) \mcl E^{(n)}(s,t_1) \frac{J(X^{(n)}(s))}{X^{(n)}(s)} ds \right| \\
& \leq \left|\int_{t_1}^{t_2} Z(s) \mcl E^{(n)} (s,t_2) \frac{J(X^{(n)}(s))}{X^{(n)}(s)} ds \right| \\
&\ \ +\left|\int_{0}^{t_1} Z(s) \mcl E^{(n)} (s,t_1)(\mcl E^{(n)}(t_1,t_2)-1) \frac{J(X^{(n)}(s))}{X^{(n)}(s)} ds \right| \\
&=H_1+H_2
\end{split}
\end{equation}
where $H_1$ and $H_2$ denote the two terms on the r.h.s. of the inequality. As for $H_1$,
by H$\ddot{o}$lder's inequality (with $1<\beta<\alpha$) and the relation $d \estn =\estn \left(-\frac{J(X^{(n)}(s))}{X^{(n)}(s)}\right)ds$, we have
\begin{equation}
\begin{split}
&  \E \left|\int_{t_1}^{t_2} Z(s) \mcl E^{(n)} (s,t_2) \frac{J(X^{(n)}(s))}{X^{(n)}(s)} ds \right| \\
& \leq \E \left[\sup_{t_1 \leq s \leq t_2} |Z(s)| \left|\int_{t_1}^{t_2} \mcl E^{(n)} (s,t_2) \left(-\frac{J(X^{(n)}(s))}{X^{(n)}(s)} \right) ds \right| \right] \\
& \leq C(\beta,T) \left\{\E \left|\int_{t_1}^{t_2} \mcl E^{(n)} (s,t_2) \left(-\frac{J(X^{(n)}(s))}{X^{(n)}(s)} \right) ds \right|^{\frac{\beta}{\beta-1}}\right\}^{\frac{\beta-1}{\beta}} \\
&=C(\beta,T) \left\{\E \left|\mcl E^{(n)} (t_1,t_2)-1 \right|^{\frac{\beta}{\beta-1}} \right\}^{\frac{\beta-1}{\beta}}.
\end{split}
\end{equation}
To estimate the expectation in the last line, we split the sample space $\Omega$
into two pieces
\begin{align*}
\Omega_1=\left\{\omega; \left|\int_{t_1}^{t_2} \frac{J(X(s))}{X(s)} ds\right| \leq (t_2-t_1)^{1/\alpha} \right \}\\
\Omega_2=\left \{\omega; \left|\int_{t_1}^{t_2} \frac{J(X(s))}{X(s)} ds \right| \geq (t_2-t_1)^{1/\alpha} \right\},
\end{align*}
and easily get
\begin{equation*}
\begin{split}
\E\left(\left|\mcl E^{(n)} (t_1,t_2)-1 \right| 1_{\Omega_1}\right)^{\frac{\beta}{\beta-1}}& \leq \E \left(\int_0^1 \left|e^{\lambda \int_{t_1}^{t_2} \frac{J(X(s))}{X(s)} ds}\right| \left|\int_{t_1}^{t_2} \frac{J(X(s))}{X(s)} ds \right| 1_{\Omega_1} d \lambda \right)^{\frac{\beta}{\beta-1}} \\
& \leq C(N)|t_2-t_1|^{\frac{\beta}{(\beta-1)\alpha}}.
\end{split}
\end{equation*}
As for the piece $\Omega_2$, by its definition and the pigeon hole principle, for each $\omega \in \Omega_2$, there exists some $r \in (t_1,t_2)$ so that $\left|\frac{J(X(r,\omega))}{X(r,\omega)}\right| \geq (t_2-t_1)^{\frac 1\alpha-1}$,
by the growth condition of $J$, we have $|X(r,\omega)| \geq |t_2-t_1|^{\frac{1-\alpha}{\kappa \alpha}}$, hence
$$\Omega_2 \subset \left \{\omega: \sup_{0 \leq t \leq T} |X(t,\omega)| \geq |t_2-t_1|^{\frac{1-\alpha}{\kappa \alpha}} \right \}.$$
By \eqref{e:UniBouT} and Chebyshev inequality, we have
$$\mathbb P (\Omega_2) \leq C(T,|x|,N)|t_2-t_1|^{\frac{\alpha-1}{\kappa \alpha}}$$
and thus
\begin{equation*}
\E \left(\left|1-\mcl E^{(n)}(t_1,t_2) \right| 1_{\Omega_2}\right)^{\frac{\beta}{\beta-1}} \leq C(T,|x|,N,\beta) |t_2-t_1|^{\frac{\alpha-1}{\kappa \alpha}}.
\end{equation*}
Combining the estimates on $\Omega_1$ and on $\Omega_2$, we immediately have
$$\E H_1 \leq C(T,|x|,N, \beta) |t_2-t_1|^{\frac{(\alpha-1)(\beta-1)}{\kappa \alpha \beta}}.$$

By some arguments as in $H_1$, $H_2$ can be estimated by
\begin{equation}
\begin{split}
& \ \E \left|\int_{0}^{t_1} Z(s) \mcl E^{(n)} (s,t_1)(1-\mcl E^{(n)}(t_1,t_2)) \frac{J(X^{(n)}(s))}{X^{(n)}(s)} ds \right| \\
& \leq \E \left[\sup_{0 \leq s \leq t_1} |Z(s)| \left|\int_{0}^{t_1} \mcl E^{(n)} (s,t_1) \left(-\frac{J(X^{(n)}(s))}{X^{(n)}(s)}\right) (1-\mcl E^{(n)}(t_1,t_2)) ds \right| \right] \\
& \leq C(\beta,T, N) \left\{\E \left[\left|\int_{0}^{t_1} \mcl E^{(n)} (s,t_1) \left(-\frac{J(X^{(n)}(s))}{X^{(n)}(s)} \right)ds \right| \left|\mcl E^{(n)}(t_1,t_2)-1\right| \right]^{\frac{\beta}{\beta-1}}\right\}^{\frac{\beta-1}{\beta}} \\
&\leq C(\beta,T,N) \left\{\E \left|\mcl E^{(n)} (t_1,t_2)-1 \right|^{\frac{\beta}{\beta-1}} \right\}^{\frac{\beta-1}{\beta}} \\
&\leq C(T,\beta,N,|x|) |t_2-t_1|^{\frac{(\alpha-1)(\beta-1)}{\kappa \alpha \beta}}
\end{split}
\end{equation}
Collecting the estimates of $\E H_1$ and $\E H_2$, we have
$$\E A_3 \leq C(T,\beta,N,|x|) |t_2-t_1|^{\frac{(\alpha-1)(\beta-1)}{\kappa \alpha \beta}}.$$
\ \\
\indent $\E A_1$ and $\E A_2$ have a similar estimates by the same arguments. Finally, by \eqref{e:Xt2-Xt1}, we have some positive constant $\delta>0$ so that
\begin{equation*}
\E |X^{(n+1)}(t_2)-X^{(n+1)}(t_1)| \leq C(T,\beta,N,|x|) |t_2-t_1|^\delta.
\end{equation*}
This concludes the proof of \eqref{e:EquCon}.
\end{proof}
% ------------------------------------------------------------------------
%GATHER{Xbib.bib}   % For Gather Purpose Only
%GATHER{Paper.bbl}  % For Gather Purpose Only
\bibliographystyle{amsplain}
%\bibliography{07degenerate}
%\begin{thebibliography}{99}

\end{document}